\documentclass[11pt,a4paper]{article}
\usepackage{amsmath,amssymb,amsthm,fullpage,graphicx}

\newtheorem{thm}{Theorem}[section]
\newtheorem{lemma}[thm]{Lemma}

\newtheorem{cor}[thm]{Corollary}

\newtheorem{defn}[thm]{Definition}
\newtheorem{rem}[thm]{Remark}

\numberwithin{equation}{section}

\allowdisplaybreaks

\begin{document}

\title{Heat kernel estimates for FIN processes\\associated with resistance forms}
\author{D.~A.~Croydon\footnote{Department of Statistics, University of Warwick, Coventry, CV4 7AL, United Kingdom.\newline  \tt{d.a.croydon@warwick.ac.uk}}, ~
  B.~M.~Hambly\footnote{Mathematical Institute, University of Oxford, Radcliffe Observatory Quarter, Woodstock Road, Oxford OX2 6GG, United Kingdom. \tt{hambly@maths.ox.ac.uk}} ~
and ~{T.~Kumagai}\footnote{Research Institute for Mathematical Sciences, Kyoto University, Kyoto 606-8502, Japan.\newline \tt{kumagai@kurims.kyoto-u.ac.jp}}}

\maketitle

\begin{abstract}
Quenched and annealed heat kernel estimates are established for Fontes-Isopi-Newman (FIN) processes on spaces equipped with a resistance form. These results are new even in the case of the one-dimensional FIN diffusion, and also apply to fractals such as the Sierpinski gasket and carpet.
\end{abstract}



\section{Introduction}

The Fontes-Isopi-Newman (FIN) diffusion is the time-change of one-dimensional Brownian motion by the positive continuous additive functional with Revuz measure given by
\begin{equation}\label{nudef}
\nu(dx)=\sum_i v_i\delta_{x_i}(dx),
\end{equation}
where $(v_i,x_i)_{i\in\mathbb{N}}$ is the Poisson point process on $(0,\infty)\times \mathbb{R}$ with intensity $\alpha v^{-1-\alpha}dvdx$ for some $\alpha\in (0,1)$, and $\delta_{x_i}$ is the probability measure placing all its mass at $x_i$. This process, introduced in \cite{FIN}, arises naturally as the scaling limit of the one-dimensional Bouchaud trap model \cite{BC,FIN} and the constant speed random walk amongst heavy-tailed random conductances in one dimension \cite{CernyEJP}. In the recent work \cite{CHKtime}, the definition of a FIN diffusion and the latter scaling results were extended to more general spaces admitting a point recurrent diffusion, namely spaces equipped with a resistance form (for a definition of such, see Section \ref{framesec}). Spaces in this class include one-dimensional Euclidean space and various fractals, such as the Sierpinski gasket and Sierpinski carpet. In the present article, we establish quenched (for typical realisation of the FIN measure) and annealed (averaged over the FIN measure) heat kernel estimates for FIN processes associated with resistance forms. These results are new even in the one-dimensional case. En route, we also extend the one-dimensional exit time bounds of \cite{Cabezas,CernyCMP} to our more general setting.

We now introduce the main objects of interest in this study. A resistance metric on a space $F$ is a function $R:F\times F\rightarrow \mathbb{R}$ such that, for every finite $V\subseteq F$, one can find a weighted graph with vertex set $V$ (here, `weighted' means edges are equipped with conductances) for which $R|_{V\times V}$ is the associated effective resistance; this definition was introduced by Kigami in the study of analysis on low-dimensional fractals, see \cite{kig1} for background. We write $\mathbb{F}$ for the collection of quadruples of the form $(F,R,\mu,\rho)$, where: $F$ is a non-empty set; $R$ is a resistance metric on $F$ such that closed bounded sets in $(F,R)$ are compact (note this implies $(F,R)$ is complete, separable and locally compact); $\mu$ is a locally finite, non-atomic
Borel regular measure of full support on $(F,R)$; and $\rho$ is a marked point in $F$. Note that the resistance metric is associated with a so-called `resistance form' $(\mathcal{E},\mathcal{F})$ (another concept introduced by Kigami), and we will further assume that for elements of $\mathbb{F}$ this form is `regular' (see Definition \ref{regulardef}). Whilst we postpone precise definitions for this terminology until Section \ref{framesec}, we note that it ensures the existence of a related regular Dirichlet form $(\mathcal{E},\mathcal{D})$ on $L^2(F,\mu)$, which we suppose is recurrent, and also a Hunt process $((X_t)_{t\geq 0},(P_x)_{x\in F})$ that admits jointly measurable local times $(L_t(x))_{x\in F,t\geq 0}$.

In our construction of a FIN process on $F$, the process $X$ introduced in the previous paragraph will play the role of Brownian motion. To expand on this, first suppose that $\nu$ is the natural generalisation of (\ref{nudef}) given by setting $\nu(dx)=\sum_i v_i\delta_{x_i}(dx)$, where now $(v_i,x_i)_{i\in\mathbb{N}}$ is the Poisson point process on $(0,\infty)\times F$ with intensity $\alpha v^{-1-\alpha}dv\mu(dx)$ for some $\alpha\in (0,1)$. We will write $\mathbf{P}$ for the probability measure on the space upon which this Poisson process is built, and observe that $\mathbf{P}$-a.s.\ the measure $\nu$ is itself a locally finite Borel regular measure of full support on $(F,R)$. Given a realisation of $\nu$ satisfying the latter properties, we then define
\[A^\nu_t:=\int_F L_t(x)\nu(dx),\]
and its right-continuous inverse $\tau^\nu(t):=\inf\{s>0:\:A_s^\nu>t\}$. The process $X^\nu$ obtained by setting
\[X^\nu_t=X_{\tau^\nu(t)}\]
is then the $\alpha$-FIN process associated with the space $(F,R,\mu)$; in general, this process might not be a diffusion, hence we call it a FIN process, rather than a FIN diffusion. (Note that, by applying the trace theorem of \cite[Theorem 6.2.1]{FOT}, this could alternatively simply be seen as Brownian motion on $(F,R,\nu)$). The quenched law of $X^\nu$ started from $x$ will be denoted by $P^\nu_x$ (i.e.\ this is the law of $X^\nu$ under $P_x$ for the given realisation of $\nu$). We have from \cite[Theorem 10.4]{Kig} that, for $\mathbf{P}$-a.e.\ realisation of $\nu$, $X^\nu$ admits a jointly continuous transition density $(p_t^\nu(x,y))_{x,y\in F,\:t>0}$; we call the latter object the quenched heat kernel for $X^\nu$, and its expectation under $\mathbf{P}$ the annealed heat kernel. Providing estimates for $(p_t^\nu(x,y))_{x,y\in F,\:t>0}$ is the goal of this article.

For convenience when presenting our results, we next state several assumptions that we will require concerning the underlying metric measure space. Towards setting out the first of these conditions, we define
$B_R(x,r):=\{y\in F:\: R(x,y)<r\}$ to be the open ball in $(F,R)$ of radius $r$, centred at $x$.
We also write $R_F:=\sup_{x,y\in F}R(x,y)$ for the diameter of $F$ with respect to $R$. Moreover, if we write $f\asymp g$ for two strictly positive functions defined on the same space, we mean that there exist constants $c_1, c_2\in (0,\infty)$ such that $c_1f\leq g\leq c_2g$ everywhere.
\begin{description}
  \item[Uniform volume growth with volume doubling (UVD)] There exist constants $c_d,c_l,c_u$ and a non-decreasing function $v:(0,\infty)\rightarrow(0,\infty)$ satisfying $v(2r) \leq c_d v(r)$ for all $r\in(0,R_F+1)$ such that
      \[c_l v(r) \leq \mu(B_R(x,r)) \leq c_u v(r), \qquad \forall x\in F,\:r\in (0,R_F+1).\]
  \item[Metric comparison (MC)] The function $d:F\times F\rightarrow \mathbb{R}$ is a metric on $F$ such that $d\asymp R^\beta$ for some $\beta>0$.
  \item[Geodesic metric comparison (GMC)] MC holds and also $d$ is a geodesic metric.
\end{description}
The most important of these conditions for our arguments is UVD; indeed, we appeal to it in all that follows. It is easily checked in the case when we have polynomial volume growth, i.e.\ $v(r)\asymp r^{\delta_f}$ for some $\delta_f>0$. The condition MC is clearly always satisfied by taking $d=R$ and $\beta=1$. However, it is often useful in examples to consider an alternative metric to the resistance metric, and so we include this as an option under this assumption. Condition GMC is relatively strong, but is applied to establish matching upper and lower annealed heat kernel bounds, and can be checked for various models of fractal (as we describe below).

We are now in a position to state our annealed heat kernel bounds. For this, we introduce the notation $h(r):=rv(r)^{1/\alpha}$, which gives a natural time-scaling for $X^\nu$. In the following result we assume $X$ is a diffusion. Note that this restriction is not needed for the on-diagonal bounds.

{\thm\label{annthm} There exists an $\alpha_c \in (0,1)$ such that for $\alpha>\alpha_c$ we have the following.\\
(a) Suppose UVD and MC, and that $X$ is a diffusion. Then there exist constants $a,c_1,c_2$ such that
\[\mathbf{E}\left(p_t^\nu(x,y)\right)\leq \frac{c_1h^{-1}(t)}{t}e^{-c_2N(a)},\qquad \forall x,y\in F,\:t\in (0,h(R_F)),\]
where
\begin{equation}\label{nadef}
N(a):=\sup\left\{n\geq 1:\: \frac{at}{n}\leq h\left(\left(\frac{d(x,y)}{n}\right)^{1/\beta}\right)\right\}.
\end{equation}
NB. If the defining set is empty, we set $N(a)=0$.\\
(b) Suppose UVD and GMC hold, and that $X$ is a diffusion. Then there exist constants $a,c_1,c_2$ such that
\[\mathbf{E}\left(p_t^\nu(x,y)\right)\geq \frac{c_1h^{-1}(t)}{t}e^{-c_2N(a)},\qquad \forall x,y\in F,\:t\in (0,h(R_F)),\]
where $N(a)$ is again defined as at (\ref{nadef}).}
\bigskip

To illustrate the above result, suppose GMC holds, and moreover the underlying space satisfies $\mu(B_d(x,r))\asymp r^{d_f}$ (where $B_d(x,r)$ is a ball with respect to the metric $d$), so that UVD holds with $v(r)=r^{\beta d_f}$. In this case the estimates have the
standard sub-diffusive form in that there exist constants $c_1,c_2,c_3,c_4$ such that: provided $\alpha>\alpha_c$, for all $x,y\in F$, $t\in (0,R_F^{1+\beta d_f/\alpha})$,
\begin{equation}\label{twosidedpoly}
c_1t^{-d_s/2}\exp\left\{-c_2\left(\frac{d(x,y)^{d_w}}{t}\right)^{\frac{1}{d_w-1}}\right\}
\leq \mathbf{E}\left(p_t^\nu(x,y)\right)\leq
c_3t^{-d_s/2}\exp\left\{-c_4\left(\frac{d(x,y)^{d_w}}{t}\right)^{\frac{1}{d_w-1}}\right\},
\end{equation}
where
\[d_w:=\frac{d_f}{\alpha}+\frac{1}{\beta}\]
can be considered to be the walk dimension of $X^\nu$ (with respect to $d$), and
\[d_s:=\frac{2d_f}{\alpha d_w}\]
can be considered to be the spectral dimension of $X^\nu$. In this case, we can take (see \eqref{alpha0def})
\[\alpha_c=\frac{\sqrt{\beta^2 d_f^2+4\beta d_f}-\beta d_f}{2}.\]
In particular, all the above assumptions hold (with $\beta=d_f=1$) in the case when $F=\mathbb{R}$, $R=d$
is the Euclidean metric on $\mathbb{R}$, and $\mu$ is Lebesgue measure on $\mathbb{R}$; this setting corresponds
to the original FIN diffusion. Hence in this case, if $\alpha>(\sqrt{5}-1)/2\approx0.618$, we obtain annealed heat kernel
estimates with
\[d_w=\frac{1+\alpha}{\alpha},\qquad d_s=\frac{2}{1+\alpha}.\]
(Note these exponents are continuous as $\alpha \rightarrow 1^-$, with limits 2 and 1, respectively, which are the usual exponents for Brownian motion; that the law of the FIN diffusion converges as $\alpha \rightarrow 1^-$ to that of Brownian motion was checked in \cite[Remark 4.4]{CroyMuir}.)

{\rem {\rm In the polynomial growth case (i.e.\ $v(r)=r^{\beta\delta_f}$), it is possible to check that the value of $\alpha_c$ given above can not be improved by the current arguments, in that the upper bound explodes for any smaller value of $\alpha$. The issue we encounter is that our techniques do not allow us to check the integrability of the on-diagonal part of the heat kernel. Note that if we consider the related issue of estimating the tail of the exit time distribution, a problem for which integrability is not a problem, then we no longer need to restrict the range of $\alpha$. For details, see Proposition \ref{annexit}, which generalises the one-dimensional results of \cite{Cabezas, CernyCMP}.
Furthermore, for general $\alpha$, the argument of Theorem \ref{annthm}(a) will show $\mathbf{E} p_t^{\nu}(\rho,\rho)^\theta\leq c (h^{-1}(t)/t)^\theta$ for suitably small $\theta$. We leave it as an open question to check the finiteness of the annealed on-diagonal heat kernel when $\alpha\leq\alpha_c$. A similar issue was encountered in the study of random walk on infinite variance Galton-Watson trees in \cite{CroydonKumagai}.}}

\bigskip

Going beyond one dimension, one might consider the example of the Sierpinski gasket. Specifically, this is the unique non-empty compact set $F\subseteq \mathbb{R}^2$ satisfying $F=\cup_{i=1}^3\psi_i(F)$, where
\[\psi_i(x):=p_i+\frac{x-p_i}{2},\qquad x\in\mathbb{R}^2,\]
and $\{p_1,p_2,p_3\}$ are the vertices of an equilateral triangle of unit side length. This is equipped with an intrinsic geodesic metric $d$ (which is equivalent to the Euclidean), and also a resistance metric $R$ that satisfies $d\asymp R^\beta$ with $\beta=\ln (2)/\ln (5/3)$, see e.g. \cite{Str} (1.6.10). Moreover, if $\mu$ is the $\ln(3)/\ln(2)$-dimensional Hausdorff measure on $F$ (with respect to $d$), then $\mu(B_d(x,r))\asymp r^{d_f}$ with $d_f=\ln(3)/\ln(2)$.
It follows that (\ref{twosidedpoly}) holds for $\alpha>\alpha_c=0.743$ with
\[d_w:=\frac{\ln\left(5\times 3^{\frac{1}{\alpha}-1}\right)}{\ln(2)},\qquad d_s:=\frac{2\ln(3)}{\ln\left(5\times 3^{\frac{1}{\alpha}-1}\right)}.\]
(Again, these exponents are continuous as $\alpha \rightarrow 1^-$, with limits being equal to the Brownian motion exponents, and a similar argument to the one-dimensional case \cite[Remark 4.4]{CroyMuir} could be used to establish the corresponding convergence of processes.)
We note that it would in fact be possible to check all the relevant conditions for the entire class of nested fractals, of which the Sierpinski gasket is just one example. The results also apply to the two-dimensional Sierpinski carpet, where to establish GMC the results of \cite{BBRes} can be applied as in \cite{CroyLT}.

As well as establishing annealed heat kernel estimates, we investigate the quenched behaviour of the heat kernel. In particular, we study the short-time asymptotics of the on-diagonal part of the heat kernel, both uniformly over compacts and pointwise. Strikingly, in both cases, the fluctuations above the mean behaviour are much smaller than those below. Whilst we postpone the most general statements of our results until later in the article (see Section \ref{qodsec}), let us briefly describe the situation when UVD holds with $v(r)\asymp r^{\delta_f}$. For any compact $G\subseteq F$ with $\mu(G)>0$, we then $\mathbf{P}$-a.s.\ have that
\begin{equation}\label{qodpoly0}
0< \limsup_{t\to 0} \frac{\sup_{x\in G}p_t^{\nu}(x,x)}{ t^{-\delta_f/(\delta_f+\alpha)}\left|\log{t}\right|^{(1-\alpha)/(\delta_f+\alpha)}}
 <\infty,
 \end{equation}
\begin{equation}\label{qodpoly05}
0<c_1\leq \inf_{x\in G}p^\nu_t(x,x)\leq c_2<\infty,\qquad \forall t\in (0,1),
\end{equation}
for some (random) constants $c_1,c_2$. Thus we see logarithmic fluctuations above the mean, and polynomial ones below. The former effect is due to points of unusually low mass, and is common for random self-similar fractals, cf. \cite{Cro2}. The latter effect is due to the atoms in the measure, at which the heat kernel remains bounded as $t\rightarrow 0$. For the pointwise results, we consider the behaviour at the distinguished point $\rho$, though the results could alternatively be stated for $\mu$-a.e.\ point; in either case, there will $\mathbf{P}$-a.s.\ not be an atom at the point under consideration. We have $\mathbf{P}$-a.s.\ that
\begin{equation}\label{qodpoly1}
0< \limsup_{t\to 0} \frac{p_t^{\nu}(\rho,\rho)}{ t^{-\delta_f/(\delta_f+\alpha)}\left(\log{|\log{t}|}\right)^{(1-\alpha)/(\delta_f+\alpha)}}
<\infty.
\end{equation}
Moreover, it $\mathbf{P}$-a.s.\ holds that, for any $\varepsilon>0$, there exists a constant $c_3$ such that
\begin{equation}\label{qodpoly2}
\liminf_{t\to 0} \frac{p_t^{\nu}(\rho,\rho)}{t^{-\delta_f/(\delta_f+\alpha)}|\log{t}|^{-3(1+\varepsilon)/\alpha}}\geq c_3,
\end{equation}
and also there is a constant $c_4$ such that
\begin{equation}\label{qodpoly3}
\liminf_{t\to 0} \frac{p_t^{\nu}(\rho,\rho)}{t^{-\delta_f/(\delta_f+\alpha)}|\log{t}|^{-1/(\delta_f+\alpha)}} \leq c_4.
\end{equation}
The asymmetry of log-logarithmic fluctuations above the mean and logarithmic fluctuations below stems from a similar asymmetry in the FIN measure, which we will derive from classical results about the fluctuations of a related $\alpha$-stable process.

Finally we also give quenched off-diagonal estimates in one-dimension. We show that in this case the fixed environment induces averaging, so that there are no oscillations in the off-diagonal terms. This is demonstrated
in establishing Theorem~\ref{qexit}, a quenched version of the results of
{\v{C}}ern{\'y} \cite{CernyCMP} and Cabezas \cite{Cabezas} on the tail of the exit time distribution,  and we then extend this to a full heat kernel estimate in the following result. Note that the integrability issues arising in the annealed case do not affect the quenched heat kernel bounds which are established for all $0<\alpha<1$.

\begin{thm}\label{thm:fullqhk} Suppose $(p^{\nu}_t(x,y))_{x,y\in\mathbb{R},\:t>0}$ is the quenched heat kernel of the one-dimensional FIN diffusion. Let $x,y$ be fixed with $|x-y|=D$. For any $\varepsilon>0$, there exist constants $c_i, i=1,\dots,4,$ such that $\mathbf{P}$-a.s.\ there exists a $t_0>0$ such that for $0<t<t_0$:
\begin{eqnarray*}
p^{\nu}_t(x,y)&\geq &c_1 t^{-d_s/2}|\log{t}|^{-3(1+\varepsilon)/\alpha} \exp\left(-c_2 \left(\frac{D^{1+1/\alpha}}{T}\right)^{\alpha} \right), \\
p^{\nu}_t(x,y)  &\leq& c_3 t^{-d_s/2}(\log{|\log{t}|})^{(1-\alpha)/\alpha} \exp\left(-c_4 \left(\frac{D^{1+1/\alpha}}{T}\right)^{\alpha} \right).
\end{eqnarray*}
\end{thm}

The remainder of the article is organised as follows. In Section \ref{framesec} we provide further background about the resistance form setting in which we are working. In Section \ref{volumesec} we establish various estimates of the masses of resistance balls with respect to the FIN measure, which will be a key ingredient for our heat kernel estimates. In Section \ref{qodsec}, we deduce our quenched on-diagonal heat kernel estimates, which yield the results at (\ref{qodpoly0})-(\ref{qodpoly3}). Section \ref{annsec} contains the proof of the annealed heat kernel bounds contained in Theorem \ref{annthm}. Since the arguments used to prove these are closely related to the proofs of exit time bounds for $X^\nu$, we also include annealed exit time bounds in this section, which extend the previously established results for the one-dimensional FIN diffusion to our more general setting. Finally, in Section \ref{qoffd} we study the quenched behaviour of the off-diagonal part of the heat kernel
in one dimension.

\section{Framework}\label{framesec}

In \cite{Kig} the notion of a resistance form was introduced to capture a natural class of objects in which the electrical resistance is a metric and the associated diffusions are point recurrent.

\begin{defn} [{\cite[Definition 3.1]{Kig}}]
Let $F$ be a non-empty set. A pair $(\mathcal{E},\mathcal{F})$ is called a \emph{resistance form} on $F$ if it satisfies the following five conditions.
\begin{description}
  \item[RF1] $\mathcal{F}$ is a linear subspace of the collection of functions $\{f:F\rightarrow\mathbb{R}\}$ containing constants, and $\mathcal{E}$ is a non-negative symmetric quadratic form on $\mathcal{F}$ such that $\mathcal{E}(f,f)=0$ if and only if $f$ is constant on $F$.
  \item[RF2] Let $\sim$ be the equivalence relation on $\mathcal{F}$ defined by saying $f\sim g$ if and only if $f-g$ is constant on $F$. Then $(\mathcal{F}/\sim,\mathcal{E})$ is a Hilbert space.
  \item[RF3] If $x\neq y$, then there exists a $f\in \mathcal{F}$ such that $f(x)\neq f(y)$.
  \item[RF4] For any $x,y\in F$,
  \begin{equation}\label{resdef}
  R(x,y):=\sup\left\{\frac{\left|f(x)-f(y)\right|^2}{\mathcal{E}(f,f)}:\:f\in\mathcal{F},\:\mathcal{E}(f,f)>0\right\}<\infty.
  \end{equation}
  \item[RF5] If $\bar{f}:=(f \wedge 1)\vee 0$, then $\bar{f}\in\mathcal{F}$ and $\mathcal{E}(\bar{f},\bar{f})\leq\mathcal{E}({f},{f})$ for any $f\in\mathcal{F}$.
\end{description}
\end{defn}

The function $R(x,y)$ defined in \eqref{resdef} can be rewritten as
\[
R(x,y)=\left(\inf\left\{\mathcal{E}(f,f):\:
f\in \mathcal{F},\: f(x)=1,\: f(y)=0\right\}\right)^{-1},
\]
which is the effective resistance between $x$ and $y$.
This is a metric on $F$ \cite[Proposition 3.3]{Kig}, which we call the \emph{resistance metric} associated with
the form $(\mathcal{E},\mathcal{F})$. We define the open ball centred at $x$ with radius $r$ in the resistance metric by $B_R(x,r):=\left\{y\in F:\:R(x,y)<r\right\}$. Throughout the paper we assume that we have a non-empty set $F$ equipped with a resistance form $(\mathcal{E},\mathcal{F})$ such that the closure of $B_R(x,r)$, denoted $\bar{B}_R(x,r)$, is compact for any $x\in F$ and $r>0$. (Note the latter condition ensures $(F,R)$ is complete, separable and locally compact.) We will also restrict our attention to resistance forms that are regular in the following sense.

\begin{defn} [{\cite[Definition 6.2]{Kig}}]\label{regulardef}
Let $C_0(F)$ be the collection of compactly supported, continuous (with respect to $R$) functions on $F$, and $\|\cdot\|_F$ be the supremum norm for functions on $F$. A resistance form $(\mathcal{E},\mathcal{F})$ on $F$ is called \emph{regular} if and only if $\mathcal{F}\cap C_0(F)$ is dense in $C_0(F)$ with respect to $\|\cdot\|_F$.
\end{defn}

We state two fundamental results that we use in a few places. Firstly, a simple consequence of \eqref{resdef},
is the H\"older continuity of functions in the domain of the Dirichlet form;
\begin{equation}\label{eq:hcont}
 |f(x)-f(y)|^2 \leq R(x,y) \mathcal{E}(f,f), \;\; \forall x,y\in F, \;\; f\in \mathcal{F}.
\end{equation}
We also recall that from \cite[Theorem 10.4]{Kig} that for $\mathbf{P}$-a.e.\ realisation of $\nu$ we have a jointly continuous heat kernel $(p^\nu_t(x,y))_{x,y\in F,t>0}$. This satisfies the following bound, which is a simple modification of \cite{Barlow}~(4.17),
\begin{equation}\label{eq:dformhkbd}
\mathcal{E}(p_t(x,\cdot),p_t(x,\cdot)) \leq \frac{p_t(x,x)}{t}.
\end{equation}

We conclude this section by noting that the doubling property of $v$ implies that we have
a constant $c>0$ such that
\begin{equation}\label{eq:polygrowth}
v(r)\geq c r^\gamma, \;\; \forall r\in (0,R_F+1).
\end{equation}
Here $\gamma=\log c_d/\log 2$, where $c_d$ is the constant appearing in the definition of UVD.

\section{Volume growth estimates}\label{volumesec}

Before proceeding to study the heat kernel, in this section we explore the behaviour of the FIN measure. Throughout we suppose that UVD holds. The FIN measure $\nu$ is closely related to an $\alpha$-stable L\'evy process, and we will use this connection to provide estimates on the local and uniform volume growth. We write $V(x,r) = \nu(B_R(x,r))$ for the volume growth function of balls in the resistance metric under the FIN measure. In the following, we let $\mathcal{L}$ be an $\alpha$-stable subordinator, and recall that we can construct this by setting $\mathcal{L}_t=\sum_{i}v_i\mathbf{1}_{\{t_i\leq t\}}$, where $(v_i,t_i)$ are the points of a Poisson process on $(0,\infty)\times \mathbb{R}_+$ with intensity $\alpha v^{-\alpha-1}dvdt$.

\begin{lemma}\label{lem:nuL} It is possible to couple $(\mathcal{L}_t)_{t\geq 0}$ and $(V(\rho,r))_{r\geq 0}$ so that, $\mathbf{P}$-a.s.,
\[ \mathcal{L}_{c_l v(r)} \leq V(\rho,r) \leq \mathcal{L}_{c_u v(r)},\qquad \forall r\in(0,R_F+1).\]
\end{lemma}

\begin{proof}
By definition we have $V(\rho,r) = \int_{B_R(\rho,r)} \sum_i v_i \delta_{x_i}(dy)$,
where the points $(v_i,x_i)$ are a Poisson point process of intensity $\alpha v^{-\alpha-1}dv\mu(dx)$.
Let $\hat{\mu}$ be a measure on $\mathbb{R}_+$ given by $\hat{\mu}([0,s)) = \mu(B_R(\rho,s))$ for all $s>0$. Thus, by projecting the points $x_i\in F$ to $\tilde{x}_i = R(\rho,x_i)\in \mathbb{R}_+$, we have $V(\rho,r) = \int_0^r \sum_i v_i \delta_{\tilde{x}_i}(dy)$, where the points  $(v_i,\tilde{x}_i)$ are a Poisson point process of intensity
$\alpha v^{-\alpha-1}dv\hat{\mu}(d\tilde{x})$. Making the change of variables
$t_i=\hat{\mu}^{-1}([0,\tilde{x}_i))$, and noting that $(v_i,t_i)$ are Poisson points with intensity
$\alpha v^{-\alpha-1}dvdt$, this implies
\[V(\rho,r) = \int_0^r \sum_i v_i \delta_{\hat{\mu}(t_i)}(ds)
= \int_0^{\hat{\mu}(r)} \sum_i v_i \delta_{t_i}(ds)
= \mathcal{L}_{\mu(B_R(\rho,r))}.\]
Applying the UVD assumption concludes the proof.
\end{proof}

We now recall some basic facts about the sample paths of $\alpha$-stable subordinators, which we
will subsequently use to control the volume growth of our measure. For statements and proofs see
\cite[Chapter~III.4]{Ber}. Firstly, there is an integral test for the upper bound on the behaviour
of $\mathcal{L}$ near 0 in that, $\mathbf{P}$-almost surely,
\[ \limsup_{t\to 0} \frac{\mathcal{L}_t}{h_t} = \left\{ \begin{array}{ll} \infty, & \mbox{if }
\int_0^1 h_t^{-\alpha} dt = \infty, \\
0, & \mbox{if } \int_0^1 h_t^{-\alpha} dt < \infty. \end{array} \right. \]
In particular we have for any positive $c$ that, $\mathbf{P}$-almost surely, there is an infinite
sequence of times $\{t_n\}_{n=1}^{\infty}$, with $t_n\to 0$ as $n\to\infty$, such that
\[ \mathcal{L}_{t_n} \geq c t_n^{1/\alpha} |\log{t_n}|^{1/\alpha}, \qquad \forall n\in\mathbb{N}, \]
and for any $\varepsilon>0$, there is a constant $C$ such that, $\mathbf{P}$-almost surely,
\[ \mathcal{L}_t \leq C t^{1/\alpha} |\log{t}|^{(1+\varepsilon)/\alpha}, \qquad \forall t\in (0,1).\]
For the lower bounds we have smaller fluctuations. Indeed, \cite[III~Theorem~11]{Ber} states
that, $\mathbf{P}$-almost surely,
\[ \liminf_{t\to 0} \frac{\mathcal{L}_t}{t^{1/\alpha}(\log|\log{t}|)^{1-1/\alpha}} = C_{\alpha}
(=\alpha(1-\alpha)^{(1-\alpha)/\alpha}). \]
Combining these results with Lemma~\ref{lem:nuL} we have the following lemma, which summarizes the
volume growth of balls in the FIN measure from a $\mu$-typical point, where there is no atom.

\begin{lemma}\label{lem:volgrowth}
(1) For any $\varepsilon>0$ there exists a $c>0$ such that
\[ V(\rho,r) \leq c v(r)^{1/\alpha}|\log{v(r)}|^{(1+\varepsilon)/\alpha}, \qquad \forall r<R_F,
\qquad \mathbf{P}\mbox{-a.s.} \]
(2) There is a $c>0$ and an infinite sequence $\{r_n\}$ with $r_n \to 0$ as $n\to\infty$ such that
\[ V(\rho,r_n) \geq c v(r_n)^{1/\alpha}|\log{v(r_n)}|^{1/\alpha}, \qquad \forall n\in\mathbb{N},
\qquad\mathbf{P}\mbox{-a.s.} \]
(3) There is a $c>0$ such that
\[ V(\rho,r) \geq c v(r)^{1/\alpha}(\log{|\log{v(r)}|})^{1-1/\alpha}, \qquad  \forall r<R_F,
\qquad \mathbf{P}\mbox{-a.s.} \]
(4) There is a $c$ and an infinite sequence $\{r_n\}$ with $r_n \to 0$ as $n\to\infty$ such that
\[ V(\rho,r_n) \leq c v(r_n)^{1/\alpha}(\log|\log{v(r_n)}|)^{1-1/\alpha}, \qquad \forall n\in\mathbb{N},
\qquad \mathbf{P}\mbox{-a.s.} \]
\end{lemma}

The uniform behaviour of balls is different, as the atoms play a role. We let $G\subseteq F$ be a compact subset with $\mu(G)>0$.

\begin{lemma}
There exist random constants $0<c_1,c_2$ such that
\[ c_1 \leq \sup_{x\in G} V(x,r) \leq c_2, \qquad  \forall r<R_F, \qquad  \mathbf{P}\mbox{-a.s.} \]
\end{lemma}

\begin{proof} The upper bound is clear as $\nu(G)<\infty$, $\mathbf{P}$-a.s. For the lower bound, we note that for any Poisson point $(v_i,x_i)$ with $x_i\in G$, we have $ \sup_{x\in G} V(x,r) \geq \nu(B_R(x_i,r)) \geq v_i >0$ for all $r>0$.
\end{proof}

More challenging is to estimate the uniform infimum of the volume. For this we state
the result obtained in \cite{Haw} for the left tail of the law of the one-dimensional subordinator.
For any fixed $t$, as $x\to 0$,
\begin{equation}
\mathbf{P}(\mathcal{L}_t\leq xt^{1/\alpha}) \sim C_1 x^{\alpha/(1(1-\alpha))} \exp(-C_2 x^{-\alpha/(1-\alpha)}). \label{eq:subtail}
\end{equation}
where $C_1 = (2\pi(1-\alpha)\alpha^{\alpha/(2(1-\alpha))})^{-1/2}$, and $C_2 = (1-\alpha) \alpha^{\alpha/(1-\alpha)}$.
We also remark that the upper tail has a simple upper bound in that there is a constant $C_3$ such that
\begin{equation}\label{eq:uptail}
\mathbf{P}(\mathcal{L}_t\geq xt^{1/\alpha}) \leq C_3 x^{-\alpha}, \qquad \forall t>0, x>0. \end{equation}

\begin{lemma}\label{volasymp}
There exist constants $c_1, c_2$ such that, $\mathbf{P}$-a.s.,
\[ c_1 \leq  \liminf_{r \to 0} \frac{\inf_{x \in G} V(x,r)}{v(r)^{1/\alpha}|\log{v(r)}|^{1-1/\alpha}}
 \leq  \limsup_{r \to 0} \frac{ \inf_{x \in G} V(x,r)}{v(r)^{1/\alpha}|\log{v(r)}|^{1-1/\alpha}} \leq c_2.\]
\end{lemma}

\begin{proof}
We first define the minimal number of balls in a cover of a set $A$
\[ N(A,r) = \min\left\{k:\exists (y_i)_{i=1}^k\mbox{ such that }A \subseteq \bigcup_i B_R(y_i,r)\right\},\]
and the maximal number of disjoint balls with centres in a set $A$
\[ N^d(A,r) = \max\left\{k:\exists (y_i)_{i=1}^k\mbox{ such that } y_i\in A\mbox{ and }
B_R(y_i,r)\cap B_R(y_j,r)=\emptyset, j\neq i\right\}. \]

For the lower bound, we see that from UVD,
\[ \mu(G) \leq \sum_{i=1}^{N(G,r)} \mu(B_R(y_i,r)) \leq c_u N(G,r) v(r). \]
so that $N(G,r) \geq c_u\mu(G) v(r)^{-1}$.

For the upper bound, given a maximal collection of $N^d(G,r/2)$ disjoint balls with centres at points $\{y_i\}$ in $G$, any point $x\in G$ must be within a distance $r/2$ of a ball, otherwise we could include another ball in our collection. Thus we have a cover of $G$ with $N^d(G,r/2)$ balls by using the same set of centres $\{y_i\}$ and with balls of double their radius. Hence $N(G,r) \leq N^d(G,r/2)$. Moreover, for $r<R_G:=\sup_{x,y\in G}R(x,y)$, as the measure is supported on $G$, using UVD
\[ \mu(G) \geq \sum_{i=1}^{N^d(G,r/2)}\mu(B_R(y_i,r/2)) \geq c_l N^d(G,r/2) v(r/2), \]
and so $N^d(G,r/2) \leq c_l\mu(G) v(r/2)^{-1}$.

We now establish our result. For convenience, we write $\phi(r) =v(r)^{1/\alpha}|\log{v(r)}|^{1-1/\alpha}$.
Let $A_{jk} = \{V(y_j,2^{-k})
\leq c_* \phi(2^{-k})\}$, where $c_*$ is a constant we will choose
later. By Lemma \ref{lem:nuL} and \eqref{eq:subtail}, we have
$\mathbf{P}(A_{jk}) \leq v(2^{-k})^{C_2c_*^{-\alpha/(1-\alpha)}}$ for any $j,k$.

As there are $N^d(G,2^{-k}) \geq cv(2^{-k})^{-1}$ disjoint balls of radius $2^{-k}$ in $G$, we have
\begin{eqnarray*}
\mathbf{P}\left(\inf_{x\in G} V(x,2^{-k}) > c_*\phi(2^{-k}) \right) &\leq &
\mathbf{P}\left(\inf_{i} V(y_i,2^{-k}) > c_*\phi(2^{-k})\right) \\
&=& \mathbf{P}\left(\bigcap_{j=1}^{N^d(G,2^{-k})}A_{jk}^c\right) \\
&=& \prod_{j=1}^{N^d(G,2^{-k})}(1-\mathbf{P}(A_{jk})) \\
&\leq& \exp(-cv(2^{-k})^{-1+C_2c_*^{-\alpha/(1-\alpha)}}).
\end{eqnarray*}
Thus, as $v$ grows at least polynomially \eqref{eq:polygrowth}, by choosing $c_*>C_2^{(1-\alpha)/\alpha}$ sufficiently
large, we obtain from a Borel-Cantelli argument that
\[ \limsup_{k\to\infty} \inf_{x\in G}\frac{V(x,2^{-k})}{\phi(2^{-k})} \leq c_*, \qquad \mathbf{P}\mbox{-a.s.}\]
It is easy to check from this that there is a constant $c$ such that
\[ \limsup_{r\to 0} \inf_{x\in G}\frac{V(x,r)}{\phi(r)}\leq c, \qquad  \mathbf{P}\mbox{-a.s.} \]

For the corresponding $\liminf$ result, we first note that there exists a collection of points $(y_i)_{i=1}^{N(G,2^{-k-1})}$
such that the balls $B_R(y_i,2^{-k-1})$ form a cover of $G$. For $2^{-k} \leq r <2^{-k+1}$, it holds that
\[ \inf_{x\in G} \frac{V(x,r)}{\phi(r)} \geq  \inf_{i=1,\dots,N(G,2^{-k-1})}  \frac{V(y_i,2^{-k-1})}
{\phi(2^{-k+1})}.\]
Using Lemma \ref{lem:nuL} and \eqref{eq:subtail} again, it is
straightforward to estimate
\begin{eqnarray*}
\mathbf{P}\left(\inf_{i=1,\dots,N(G,2^{-k-1})} \frac{V(y_i,2^{-k-1})}{\phi(2^{-k+1})} \leq x\right)
&\leq & \sum_{i=1}^{N(G,2^{-k-1})}
 \mathbf{P}\left( \frac{V(y_i,2^{-k-1})}{\phi(2^{-k+1})} \leq x\right) \\
&\leq & c v(2^{-k-1})^{-1+C_2x^{-\alpha/(1-\alpha)}}.
\end{eqnarray*}
As the volume function grows polynomially, provided we choose $x<C_2^{(1-\alpha)/\alpha}$, we can
sum this expression and, by Borel-Cantelli, see that there is a constant $c>0$ such that
\[ \liminf_{r\to 0} \inf_{x\in G} \frac{V(x,r)}{\phi(r)} \geq c, \qquad  \mathbf{P}\mbox{-a.s.}\]
\end{proof}

\section{Quenched on-diagonal heat kernel estimates}\label{qodsec}

We will write $P^{\nu}_\rho$ for the law of the FIN diffusion started from $\rho$ in the fixed environment $\nu$ and write $E^{\nu}_{\rho}$
for the expectation with respect to this measure. We now turn our attention to bounds for the quenched transition density $(p^\nu_t(x,y))_{x,y\in F,t>0}$, starting in this section with quenched on-diagonal estimates. In Section \ref{localsec} we derive pointwise estimates, and in Section \ref{globalsec} estimates that hold uniformly on compacts. Our arguments adapt techniques of \cite{Cro, Kum}, which develop heat kernel bounds for resistance forms.

\subsection{Local quenched heat kernel bounds}\label{localsec}

By results of \cite{Cro, Kum}, we know that, for the on-diagonal bounds of interest here, it will be sufficient to understand information about the volume growth of balls in the resistance metric. We note from \cite[Lemma 4.1]{Kum} that, for resistance balls under the UVD assumption on the base measure $\mu$, there is a constant $C_R<1$ such that for all $x\in F, r\in(0,R_F)$,
\[ C_R r \leq R(x,B_R(x,r)^c) \leq r.\]
Moreover, the argument of \cite[Proposition 4.1]{Kum} in the case of a measure satisfying the
UVD assumption is a local argument, and can be applied in our case to give the following.

\begin{lemma} \label{lem:diagub}
$\mathbf{P}$-a.s. we have
\begin{equation}
 p_{2rV(x,r)}^{\nu}(x,x) \leq \frac{2}{V(x,r)},\qquad\forall x\in F,\:r\in(0,R_F).\label{eq:hkub}
\end{equation}
\end{lemma}

The corresponding lower bound has a local version that is not so straightforward. The standard approach
to the on-diagonal lower bound is to estimate the tail of the exit time distribution from balls. In
order to do this estimates on the mean exit time are required. Let $T_A = \inf\{t>0: X_t \notin A\}$.
The arguments in \cite{Kum} yield the following.
\begin{lemma} For $\mathbf{P}$-a.e.\ realisation of $\nu$, it holds that
\begin{eqnarray*}
E_x^{\nu} T_{B_R(\rho,r)} &\leq & rV(\rho,r), \qquad \forall x\in F,\:r\in(0,R_F/2),\\
E_{\rho}^{\nu} T_{B_R(\rho,r)} &\geq & \frac{1}{2}C_R r V\left(\rho,\frac14C_R^2 r\right),\qquad \forall r\in(0,R_F/2).
\end{eqnarray*}
\end{lemma}

\begin{proof}
The main part of the argument of \cite[Proposition 4.2]{Kum} can be used as it relies on Green's
function estimates which are independent of the measure. In particular, writing $g_{B_R(\rho,r)}$
for the Green's function for the process killed upon exiting $B_R(\rho,r)$, these estimates can
be summarized as
\[ g_{B_R(\rho,r)}(\rho,\rho) \geq C_R r, \qquad  g_{B_R(\rho,r)}(x,y) \leq r,
\qquad  \forall x,y \in B_R(\rho,r),\:r\in(0,R_F/2). \]
Applying the upper bound here, we deduce that, for all $x\in F$, $r\in(0,R_F/2)$,
\[
E_x^{\nu} T_{B_R(\rho,r)} =\int_{B_R(\rho,r)} g_{B_R(\rho,r)}(x,y) \nu(dy)\leq  r V(\rho,r).\]
For the corresponding lower bound, we can follow the argument in \cite[Proposition 4.2]{Kum} to conclude that $g_{B_R(\rho,r)}(\rho,y) \geq \frac{1}{2}C_R r$ for all $y \in B_R(\rho,\frac{1}{4}C_R^2  r)$,
$r\in(0,R_F/2)$. Hence we have
\[  E_{\rho}^{\nu} T_{B_R(\rho,r)} =\int_{B_R(\rho,r)} g_{B_R(\rho,r)}(x,y) \nu(dy)\geq \frac{1}{2}C_R
r V\left(\rho,\frac{1}{4}C_R^2  r\right), \]
as desired.
\end{proof}

We now use these exit time estimates to get a local heat kernel estimate.

\begin{lemma}\label{lem:diaglb} For $\mathbf{P}$-a.e.\ realisation of $\nu$, it holds that:
for every $r\in(0,R_F/2)$ and $t \leq \frac{1}{4}C_R rV(\rho,\frac{1}{4}C_R^2  r)$,
\[ p_{t}^{\nu}(\rho,\rho) \geq \left(\frac{C_R V(\rho,\frac{1}{4}C_R^2  r)}{4V(\rho,r)}\right)^2 V(\rho,r)^{-1}. \]
\end{lemma}

\begin{proof} From the fact that
\[ E_\rho^{\nu} T_{B_R(\rho,r)} \leq t + E_\rho^{\nu}\left( \mathbf{1}_{\{T_{B_R(\rho,r)}>t\}}
E_{X_s}^{\nu}(T_{B_R(\rho,r)})\right), \]
using the estimates on the mean exit time we have
\[\frac{1}{2}C_R rV\left(\rho,\frac{1}{4}C_R^2  r\right) \leq t + rV(\rho,r)P_\rho^{\nu}(T_{B_R(\rho,r)}>t). \]
In particular, this implies
\begin{equation}\label{exittail}
 P_\rho^{\nu}(T_{B_R(\rho,r)}>t) \geq \frac{\frac{1}{2}C_R rV(\rho,\frac{1}{4}C_R^2  r)}{rV(\rho,r)} -
\frac{t}{rV(\rho,r)}\geq \frac{C_R V(\rho,\frac{1}{4}C_R^2  r)}{4V(\rho,r)}
\end{equation}
for $t \leq \frac{1}{4}C_R rV(\rho,\delta r)$. Now, by applying Cauchy-Schwarz  as in
\cite[Proposition 4.3]{Kum}, one obtains the estimate $P_x^{\nu}(T_{B_R(\rho,r)}>t)^2 \leq
p_{2t}^{\nu}(\rho,\rho) V(\rho,r)$, and so we deduce
\[p_{2t}^{\nu}(\rho,\rho) \geq \left(\frac{C_R V(\rho,\frac{1}{4}C_R^2  r)}{4V(\rho,r)}\right)^2 V(\rho,r)^{-1}. \]
\end{proof}

We next apply Lemmas~\ref{lem:diagub} and~\ref{lem:diaglb} in combination with the volume estimates from the previous
section to deduce quenched local heat kernel estimates for $X^\nu$. Our results will be stated in terms of the inverse of the
function
\[ h(r) = rv(r)^{1/\alpha}. \]
We give some straightforward properties of this function and its
inverse arising from the volume doubling property of $v(r)$. The reader may find it helpful to think of
the volume growth for the base measure as given by $v(r) = r^{\delta_f}$, which is the case for one-dimensional
Euclidean space and self-similar fractal sets.

\begin{lemma}
\label{lem:hprops}
\begin{enumerate}
\item[(1)] The function $h(r)$ is increasing and has a doubling property in that $h(2r) \leq \tilde{c} h(r)$,
for $\tilde{c}= 2c_d^{1/\alpha}>2$ where $c_d$ is the constant that appears in the definition of UVD.
\item[(2)] The function has an inverse $h^{-1}(r)$ which is an increasing function for all $r<R_F$
and satisfies the growth condition $2h^{-1}(r) \leq h^{-1}(\tilde{c} r)$.
\item[(3)] Let $q=\log{\tilde{c}}/\log{2}=1+\gamma/\alpha>1$. There is a constant $\hat{c}$ such that
$h(r)\geq \hat{c} r^q$ for $r<r_F$ and hence there is a constant $c'$ such that $h^{-1}(r) \leq c' r^{1/q}$
for all $r<r_F$.
\item[(4)] $r/h^{-1}(r) \geq r^{1-1/q}/c'$ is increasing in $r$. In particular
\begin{eqnarray*}
 c_1  h^{-1}(r) |\log{r}|^{1/q} \leq & h^{-1}(r|\log{r}|) &\leq c_2 h^{-1}(r) |\log{r}| , \;\; r<R_F, \\
c_3 h^{-1}(r) (\log{|\log{r}|})^{1/q} \leq & h^{-1}(r\log|\log{r}| ) &\leq c_4 h^{-1}(r)\log|\log{r}|,
 \;\; r<R_F.
\end{eqnarray*}
\end{enumerate}
\end{lemma}

\begin{proof}
These are easy consequences of the fact that $v(r)$ is increasing and has the volume doubling property.
\end{proof}

\begin{thm}\label{localhk}
(1) There exists a deterministic constant $c$ and a random constant $t_F$ such that
\[ p_t^{\nu}(\rho,\rho) \leq c \frac{h^{-1}(t)}{t} \left(\log{|\log{t}|}\right)^{(1-\alpha)/\alpha},
\qquad \forall t<t_F,
\qquad  \mathbf{P}\mbox{-a.s.} \]
(2) For any $\varepsilon>0$, there exists a deterministic constant $c$ and a random constant $t_F$ such that
\[ p_t^{\nu}(\rho,\rho) \geq c\frac{h^{-1}(t)}{t} |\log{t}|^{-3(1+\varepsilon)/\alpha},\qquad \forall t<t_F,
\qquad  \mathbf{P}\mbox{-a.s.} \]
(3) Also there is a random infinite sequence of times $t_n$ with $t_n\to 0$ such that
\[ p_{t_n}^{\nu}(\rho,\rho) \leq \frac{h^{-1}(t_n)}{t_n} |\log{t_n}|^{-q/\alpha}, \qquad\forall n\in\mathbb{N},
\qquad  \mathbf{P}\mbox{-a.s.} \]
\end{thm}

\begin{proof} The upper bound of (1) is a simple application of the upper heat kernel estimate
\eqref{eq:hkub} in terms of volume growth, and the lower bound on the volume growth result in
Lemma~\ref{lem:volgrowth}(3), with the properties of the function $h^{-1}$ from Lemma~\ref{lem:hprops}.
That is if $t=rV(\rho,r)$, then $p_t(\rho,\rho)\leq c r/t$. Then for $t \geq c r v(r)^{1/\alpha} |\log|\log v(r)||^{1-1/\alpha} = ch(r)|\log|\log{h(r)/r}||^{1-1/\alpha}$
we require $h(r) \leq t|\log|\log{(t/h^{-1}(t))}||^{(1-\alpha)/\alpha}$. Using the lower bound on
$t/h^{-1}(t)$ and then property (4) of Lemma~\ref{lem:hprops} gives the result.

The bound at (3) is another straightforward consequence of \eqref{eq:hkub}, the lower bound on the volume
in Lemma~\ref{lem:volgrowth}(2) and the properties of $h^{-1}$.

For the lower bound of (2), we use Lemma \ref{lem:diaglb} and
apply Lemma ~\ref{lem:volgrowth} again to deduce
that (with a modification of $\varepsilon$ in the last line)
\begin{eqnarray*}
 p_{2t}^{\nu}(\rho,\rho) &\geq&  \left(\frac{C_R V(\rho,\frac{1}{4}C_R^2  r)}{4V(\rho,r)}\right)^2 V(\rho,r)^{-1} \\
 &\geq & cv(r)^{-1/\alpha}|\log{v(r)}|^{-3(1+\varepsilon)/\alpha}(\log|\log{v(r)}|)^{2(1-1/\alpha)}\\
&\geq & c  v(r)^{-1/\alpha}|\log{v(r)}|^{-3(1+\varepsilon)/\alpha},
\end{eqnarray*}
provided $r$ is such that $t \leq \frac{1}{4}C_R  rV(\rho,r)$. This will hold if we take $r$ such that
$$t\leq c' rv(r)^{1/\alpha}
(\log|\log{v(r)}|)^{1-1/\alpha}=c'h(r)(\log|\log{(h(r)/r)}|)^{1-1/\alpha}.$$
Inverting this gives $ r \geq c'' h^{-1}(t) (\log|\log{t}|)^{(1-\alpha)/\alpha}$. Substituting this into the above bound leads to
\[ p_t^{\nu}(\rho,\rho) \geq   \frac{cr}{h(r)}\left|\log{\frac{h(r)}{r}}\right|^{-3(1+\varepsilon)/\alpha}. \]
Substituting in for $r$, using the properties of $h^{-1}$ and adjusting $\varepsilon$ gives the result.
\end{proof}

We can prove a sharper version for the upper fluctuations. In order to do this we consider a slight modification of our function
$h$ and define $h_{ll}(r) = rv(r)^{1/\alpha}(\log\log{v(r)})^{1-1/\alpha}$.

\begin{thm}\label{hll}
We have
\[ 0< \limsup_{t\to 0} \frac{p^{\nu}_t(\rho,\rho)t}{h_{ll}^{-1}(t)}  <\infty, \;\; \mathbf{P}-a.s. \]
\end{thm}

\begin{proof}
The upper bound is essentially derived in the proof of Lemma~\ref{localhk}.

For the lower bound we use Lemma~\ref{lem:diaglb}. By Lemma~\ref{lem:volgrowth}(4) we have a sequence
$\{r_n\}_{n=1}^{\infty}$ such that $V(\rho,r_n) \leq c v(r_n)^{1/\alpha}(\log|\log{v(r_n)}|)^{1-1/\alpha}$. We also know that
\[ V(\rho,\frac{1}{4}C_R^2  r_n) \geq  c' v(\frac{1}{4}C_R^2  r_n)^{1/\alpha}(\log|\log{v(\frac{1}{4}C_R^2  r_n)}|)^{1-1/\alpha}
 \geq c''  v(r_n)^{1/\alpha}(\log|\log{v(r_n)}|)^{1-1/\alpha}.\]
 Thus, almost surely, there is a sequence of times $\{t_n\}_{n=1}^{\infty}$ such that for
\[ t_n \leq \frac12 c'' r_n v(r_n)^{1/\alpha}(\log|\log{v(r_n)}|)^{1-1/\alpha}=c'' h_{ll}(r_n), \]
we have $p_{t_n}(\rho,\rho) \geq c''' v(r_n)^{-1/\alpha}(\log|\log{v(r_n)}|)^{1/\alpha-1}$. That is $p_{t_n}(\rho,\rho) \geq c''' h_{ll}^{-1}(t_n)/t_n$, which gives us the result.
\end{proof}

Finally, note that in the case where $v(r)=r^{\delta_f}$, we obtain (\ref{qodpoly1}) from Theorem~\ref{hll}, and (\ref{qodpoly2})
and (\ref{qodpoly3}) from Theorem~\ref{localhk}.

\subsection{Global quenched heat kernel estimates}\label{globalsec}

We can use the same ideas as in the previous section to obtain bounds on the on-diagonal heat kernel that are uniform on compacts.
Throughout, we let $G\subseteq F$ be a compact subset with $\mu(G)>0$. We begin with the behaviour of the infimum. The atoms of
$\nu$ result in points where the heat kernel does not diverge as $t\to 0$.

{\thm There exist random constants $c_1,c_2$ and a deterministic constant $t_F$ such that
\[ 0<c_1 \leq \inf_{x\in G} p_t^{\nu}(x,x) \leq c_2, \qquad  \forall t<t_F, \qquad  \mathbf{P}\mbox{-a.s.} \]}

\begin{proof}
By \cite{Kig} the proof of Theorem~10.4 and Lemma~10.8, we have that $p_t^{\nu}(x,x)$ is a strictly positive decreasing function of $t$ for each $x$, and so $\inf_{x\in G} p_t^{\nu}(x,x) \geq \inf_{x\in G} p_{t_F}^{\nu}(x,x)$ for $t<t_F$.
Applying the continuity of the heat kernel (see \cite{Kig}), the latter is strictly positive. Thus we have the lower bound.

For the upper bound we observe that if we take a point $(v_i,x_i)$ with $x_i\in G$ in the Poisson process, then $\inf_{x \in G} p_t^{\nu}(x,x) \leq p_t^{\nu}(x_i,x_i)$ for all $t>0$. From the local upper bound \eqref{eq:hkub} we have, as $\nu(B_R(x_i,r)) \geq c$, that $p_t^{\nu}(x_i,x_i) \leq c$, and the result follows.
\end{proof}

For the supremum of the heat kernel, we have the following estimates.

\begin{thm} \label{thm:globalub}
For any $\varepsilon>0$, there exist deterministic constants $c_1,c_2$ and a random constant $t_F$ such that
\[c_1\frac{h^{-1}(t)}{t} |\log{t}|^{-3(1+\varepsilon)/\alpha}\leq \sup_{x\in G} p_t^{\nu}(x,x) \leq c_2 \frac{h^{-1}(t)}{t}
|\log{t}|^{(1-\alpha)/\alpha},\]
for every $t<t_F$, $\mathbf{P}$-a.s.
\end{thm}

\begin{proof}
The upper bound follows from \eqref{eq:hkub} with the lower bound on the infimum of the volumes of balls.
That is if we set $t=rV(x,r)$ we have $p_t(x,x) \leq cr/t$. By choosing
\begin{equation}
 t\geq r\inf_{x\in G} V(x,r) \geq ch(r)\left|\log{\left(\frac{h(r)}{r}\right)}\right|^{1-1/\alpha}, \label{eq:globt}
\end{equation}
we have, using Lemma~\ref{lem:hprops}(4), that  $r\leq c' h^{-1}(t)|\log{t}|^{(1-\alpha)/\alpha}$.
Substituting this in for $r$ in the upper bound on $p_t(x,x)$ gives the uniform upper bound.

The lower bound is a simple consequence of Theorem \ref{localhk}(2).
\end{proof}

We can also give a fluctuation result for the supremum of the heat kernel. For this we consider the function
$h_l(r)=rv(r)^{1/\alpha}|\log{v(r)}|^{1-1/\alpha}$.

\begin{thm}
We have
\[ 0<  \limsup_{t\to 0} \frac{\sup_{x\in G} p^{\nu}_t(x,x) t}{h_l^{-1}(t)} < \infty, \;\; \mathbf{P}\mbox{-a.s.} \]
\end{thm}

\begin{proof}
For the upper bound we just make a minor modification of the proof of the upper bound in Theorem~\ref{thm:globalub}, using our
function $h_l(t)$ in \eqref{eq:globt}.

For the lower bound, by Lemma~\ref{volasymp}, there is a sequence of points and radii $\{x_n,r_n\}_{n=1}^{\infty}$
with $x_n\in F$ and $r_n\to 0$, such that $V(x_n,r_n) \leq c v(r_n)^{1/\alpha}|\log{v(r_n)}|^{1-1/\alpha}$. We also obtain by applying the lower bound in Lemma~\ref{volasymp} and then UVD,
\[ V(x_n,\frac14C_R^2 r_n) \geq c_1 v(\frac14C_R^2 r_n)^{1/\alpha}|\log{v(\frac14C_R^2 r_n)}|^{1-1/\alpha} \geq c_2 v(r_n)^{1/\alpha}|\log{v(r_n)}|^{1-1/\alpha}. \]
Thus, if we take $t_n = h_l(r_n)$ we will have a sequence of points and times, with $t_n\to 0$ such that as
$t_n = h_l(r_n) \leq \frac14C_R r_nV(x_n,r_n)$, by Lemma~\ref{lem:diaglb}, we have
\[ p_{t_n}(x_n,x_n) \geq \left( \frac{C_R V(x_n,\delta r_n)}{4V(x_n,r_n)}\right)^2 V(x_n,r_n)^{-1}\geq   \frac{cr_n}{h_l(r_n)} =  \frac{ch_l^{-1}(t_n)}{t_n}.\]
This gives the lower bound on the upper fluctuation of $\sup_{x\in G}p_t(x,x)$.
\end{proof}

Again we can specialize these results to the case where $v(r)=r^{\delta_f}$ to obtain (\ref{qodpoly0}) and (\ref{qodpoly05}).

\section{Annealed heat kernel and exit time estimates}\label{annsec}

In this section, we prove Theorem \ref{annthm}. Throughout, we suppose that UVD and MC hold, with the latter condition giving us the existence of a metric $d$ for which $d\asymp R^\beta$ for some $\beta>0$. Moreover, for the entirety of this section, we also suppose that $X$ is a diffusion. Note that this implies that, for $\mathbf{P}$-a.e.\ realisation of $\nu$, $X^\nu$ is also a diffusion. In the proofs of both the upper and lower annealed heat kernel bounds, we will apply chaining arguments, the success of which depends on exploiting the independence of $\nu$ between disjoint regions of space.

The annealed heat kernel result is also closely linked to the following exit time bound. To state this, we use the abbreviation
\[T_D:=T_{B_d(\rho,D)}\]
for the exit time of the ball $B_d(\rho,D)$ by $X^\nu$. We will also write
\[\mathbb{P}_x(\cdot):=\int P^\nu_x(\cdot)\mathbf{P}(d\nu)\]
for the annealed law of $X^\nu$ started from $x\in F$, and $D_F:=\sup_{x,y\in F}d(x,y)$ for the diameter of $F$ with respect to $d$.

{\thm\label{annexit} (a) Suppose UVD and MC hold, and that $X$ is a diffusion. Then there exist constants $a,c_1$ such that
\[\mathbb{P}_\rho\left(T_D \leq T\right)\leq  e^{-c_1N(a)},\qquad \forall D\in(0,D_F/2),\:T\in (0,h(R_F)),\]
where $N(a)$ is defined as at (\ref{nadef}) with $T$ in place of $t$ and $D$ in place of $d(x,y)$.\\
(b) Suppose UVD and GMC hold, and that $X$ is a diffusion. Then there exist constants $a,c_1,c_2$ such that
\[\mathbb{P}_\rho\left(T_D \leq T\right)\geq  c_1 e^{-c_2N(a)},\qquad \forall D\in(0,D_F/2),\:T\in (0,h(R_F)),\]
where $N(a)$ is again defined as at (\ref{nadef}) with $T$ in place of $t$ and $D$ in place of $d(x,y)$.}

\subsection{Proof of upper bounds}

We start with the proof of the upper annealed exit time bound.

\begin{proof}[Proof of Theorem \ref{annexit}(a)] Let $N=N(a)$. Clearly we can assume $N\geq1$, else the result is trivial. We also note that it is an elementary exercise to check from condition MC and the definition at (\ref{nadef}) that $N$ is finite. Set $\Delta=D/N$, and define a sequence of stopping times $(\sigma_i)_{i\geq 0}$ by setting $\sigma_0=0$, and
\[\sigma_{i+1}=\inf\left\{t>\sigma_i:\: \sup_{s\leq \sigma_i}d\left(X_s^\nu,X_t^\nu\right)\geq \Delta\right\}.\]
(Note this sequence might terminate if the space has finite diameter.) Moreover, write
\[\tilde{\sigma}_{i}=\inf\left\{t>\sigma_i:\: d\left(X_{\sigma_i}^\nu,X_t^\nu\right)\geq \Delta/2\right\},\]
so that $\sigma_0\leq \tilde{\sigma}_0\leq \sigma_1\leq \tilde{\sigma}_1\dots$. Since $X^\nu_{[0,\sigma_i)}\subseteq {B}_d(0,i\Delta)$, we have $\sigma_N\leq T_D$, and so
\begin{eqnarray*}
\mathbb{P}_\rho\left(T_D \leq T\right)&\leq &\mathbb{P}_\rho\left(\sigma_N \leq T\right)\\
&\leq&\mathbb{P}_\rho\left(\sum_{i=0}^{N-1}\left(\tilde{\sigma}_i-\sigma_i\right) \leq T\right)\\
&\leq &e^{\theta T}\mathbb{E}_\rho\left(e^{-\theta\sum_{i=0}^{N-1}\left(\tilde{\sigma}_i-\sigma_i\right)}\right)
\end{eqnarray*}
for any $\theta>0$. Now, note that we can write
\[\tilde{\sigma}_i-\sigma_i=\int_{B_d(X_{\sigma^X_i},\Delta/2)}\left(L_{\tilde{\sigma}^X_i}(x)-L_{{\sigma}^X_i}(x)\right)\nu(dx),\qquad i=0,\dots,N-1,\]
where $\tilde{\sigma}^X_i$ and ${\sigma}^X_i$ are defined similarly to $\tilde{\sigma}_i$ and ${\sigma}_i$, but with the Brownian motion $X$ in place of the FIN diffusion $X^\nu$. Since the balls
$B_d(X_{\sigma^X_i},\Delta/2)$, $i=0,\dots,N-1$, are disjoint, we thus have that, conditional on the Brownian motion $X$, the random variables $(\tilde{\sigma}_i-\sigma_i)_{i=0}^{N-1}$ are independent. In particular, this yields
\[\mathbb{P}_\rho\left(T_D \leq T\right)\leq e^{\theta T}\mathbb{E}_\rho\left(\prod_{i=0}^{N-1}
\mathbb{E}_\rho\left(e^{-\theta\left(\tilde{\sigma}_i-\sigma_i\right)}\:\vline\: X\right)\right).\]
Applying the strong Markov property for $X$ at time $\sigma^X_{N-1}$, we consequently find that
\begin{equation}\label{aftermarkov}
\mathbb{P}_\rho\left(T_D \leq T\right)\leq e^{\theta T}\mathbb{E}_\rho\left(
\prod_{i=0}^{N-2}
\mathbb{E}_\rho\left(e^{-\theta\left(\tilde{\sigma}_i-\sigma_i\right)}\:\vline\: X\right)\times \mathbb{E}_{X_{\sigma^X_{N-1}}}\left(e^{-\theta\tilde{\sigma}_0}\right)
\right).
\end{equation}

In the next part of the proof, we derive the following bound:
\begin{equation}\label{target}
\mathbb{E}_{x}\left(e^{-\theta\tilde{\sigma}_0}\right)\leq 1-c_1+\frac{c_2}{\theta h(r)}.
\end{equation}
First, recall from (\ref{exittail}) that
\[P_x^\nu\left(\tilde{\sigma}_0\leq t\right)\leq 1-\frac{c_1V(x,\frac14C_R^2 r)}{V(x,r)}+\frac{t}{rV(x,r)},\]
where we write $r=\Delta^{1/\beta}$. Integrating this (cf.\ \cite[Lemma 1.1]{BarBas}), we find
\begin{equation}\label{moment0}
E_x^\nu\left(e^{-\theta\tilde{\sigma}_0}\right)\leq 1-\frac{c_1V(x, \frac14C_R^2 r)}{V(x,r)}+\frac{1}{\theta rV(x,r)}.
\end{equation}
We next apply the coupling with a stable process of Lemma \ref{lem:nuL} to deduce that
\begin{equation}\label{moment1}
\mathbf{E}\left(\frac{V(x,\frac14C_R^2 r)}{V(x,r)}\right)\geq \mathbf{E}\left(\frac{\mathcal{L}_{c_lv( \frac14C_R^2 r)}}{\mathcal{L}_{c_uv(r)}}\right)=
\mathbf{E}\left(\frac{\mathcal{L}_{c_lv(\frac14C_R^2r)/c_u v(r)}}{\mathcal{L}_{1}}\right)\geq
\mathbf{E}\left(\frac{\mathcal{L}_{c_l/c_uc_d^{n}}}{\mathcal{L}_{1}}\right)=c_2>0,
\end{equation}
where for the first equality we apply the self-similarity under scaling of the process $\mathcal{L}$, and for the second inequality we repeatedly
apply the doubling property of $v$ and set $n=\lceil -\log{\frac14C_R^2}/\log{2}\rceil$. Furthermore, again applying Lemma \ref{lem:nuL} and the self-similarity of $\mathcal{L}$, we have that
\begin{equation}\label{moment2}
\mathbf{E}\left(\frac{1}{ rV(x,r)}\right)\leq \frac{c_3}{h(r)}\mathbf{E}\left(\frac{1}{\mathcal{L}_1}\right).
\end{equation}
From (\ref{eq:subtail}), we have that the expectation on the right-hand side here is finite. Substituting (\ref{moment1}) and (\ref{moment2}) into the $\mathbf{P}$-integrated version of (\ref{moment0}), and relabelling the constants, we obtain the bound at (\ref{target}).

Returning to (\ref{aftermarkov}), we apply the bound of the previous paragraph to deduce that
\[\mathbb{P}_\rho\left(T_D \leq T\right)\leq e^{\theta T}\mathbb{E}_\rho\left(
\prod_{i=0}^{N-2}
\mathbb{E}_\rho\left(e^{-\theta\left(\tilde{\sigma}_i-\sigma_i\right)}\:\vline\: X\right)\right)
\times \left(1-c_1+\frac{c_2}{\theta h(r)}\right).\]
Iterating the argument, this gives
\[\mathbb{P}_\rho\left(T_D \leq T\right)\leq e^{\theta T}\left(1-c_1+\frac{c_2}{\theta h(r)}\right)^{N}\leq e^{\theta T-c_1N+\frac{c_2N}{\theta h(r)}}.\]
Optimising over $\theta$ yields
\[\mathbb{P}_\rho\left(T_D \leq T\right)\leq  e^{-c_1N+2\sqrt{\frac{c_2TN}{h(r)}}}\leq e^{-c_1N+N\sqrt{\frac{4c_2}{{a}}}},\]
where the second inequality follows from our choice of $N$. Thus taking $a$ large gives the desired conclusion.
\end{proof}

To establish the on-diagonal upper bound of Theorem \ref{annthm}(a), we use the techniques developed in \cite{BarKum, Cro2}.

\begin{proof}[Proof of Theorem \ref{annthm}(a)] For the first part of the proof, we suppose that $at\leq h(d(x,y)^{1/\beta})$, where $a$ is the constant of Theorem \ref{annexit}(a), so that $N(a)\geq 1$. We decompose the heat kernel as follows:
\begin{equation}
p_t^\nu(x,y)\leq\int_{H_{y,x}}p_{t/2}^\nu(x,z)p_{t/2}^\nu(z,y)\nu(dz)+\int_{H_{x,y}}p_{t/2}^\nu(x,z)p_{t/2}^\nu(z,y)\nu(dz)=:I_1+I_2,\label{eq:hkubsplit}
\end{equation}
where $H_{x,y}:=\{z\in F:\:R(x,z)\leq R(y,z)\}$. Hence note that, by symmetry, to complete the proof it will suffice to show that
\[\mathbf{E}I_1\leq\frac{c_1h^{-1}(t)}{t}e^{-c_2N(a)}.\]
To this end, writing $B=B_R(x,\frac14 R(x,y))$, $H=H_{y,x}^c$, observe that
\begin{eqnarray*}
I_1& = &\int_{[0,t/2]\times \partial B}P_x^\nu\left(T_{B}\in ds,\:X^\nu_{T_{B}}\in dw\right) \int_{[0,t/2-s]\times \partial H}P_w^\nu\left(T_{H}\in ds',\:X^\nu_{T_{H}}\in dw'\right)\\
&&\qquad\qquad\qquad\qquad\qquad\qquad\qquad\qquad\qquad\times\int_{H_{x,y}^c}p_{t/2-s-s'}^\nu(w',z)p_{t/2}^\nu(z,y)\nu(dz)\\
&\leq&\int_{[0,t/2]\times \partial B}\hspace{-5pt}P_x^\nu\left(T_{B}\in ds,\:X^\nu_{T_{B}}\in dw\right) \int_{[0,t/2-s]\times \partial H}\hspace{-5pt}P_w^\nu\left(T_{H}\in ds',\:X^\nu_{T_{H}}\in dw'\right)p_{t-s-s'}^\nu(w',y)\\
&\leq&\int_{[0,t/2]\times \partial B}\hspace{-5pt}P_x^\nu\left(T_{B}\in ds,\:X^\nu_{T_{B}}\in dw\right) \int_{[0,t/2-s]\times \partial H}\hspace{-5pt}P_w^\nu\left(T_{H}\in ds',\:X^\nu_{T_{H}}\in dw'\right)\\
&&\qquad\qquad\qquad\qquad\qquad\qquad\qquad\qquad\qquad\qquad\qquad\times\sqrt{p_{t/2}^\nu(w',w')p_{t/2}^\nu(y,y)}\\
&\leq&\int_{\partial B}\hspace{-5pt}P_x^\nu\left(T_{B}\leq t/2,\:X^\nu_{T_{B}}\in dw\right) \int_{\partial H}\hspace{-5pt}P_w\left(X_{T_{H}}\in dw'\right)\sqrt{p_{t/2}^\nu(w',w')p_{t/2}^\nu(y,y)},\\
\end{eqnarray*}
where for the second inequality we apply Cauchy-Schwarz to obtain that $p_{t-s-s'}^\nu(w',y)\leq p_{t-s-s'}^\nu(w',w')^{1/2}p_{t-s-s'}^\nu(y,y)^{1/2}$, and apply the fact that the on-diagonal part of the heat kernel is monotonically decreasing as a function of $t$.
Moreover, for the third inequality, we use that the hitting distribution of $X^\nu$ is the same as for the Brownian motion $X$. Now, for a fixed $t$, define $\Lambda_{w',y}$ to be equal to
\[\inf\left\{\lambda>\lambda_0:\frac{v(h^{-1}(t/4\lambda))^{1/\alpha}}{(\log \lambda)^{\frac{1-\alpha}{\alpha}}} \leq V\left(z,h^{-1}(t/4\lambda)\right)\leq\lambda v(h^{-1}(t/4\lambda))^{1/\alpha},z\in\{w',y\}\right\},\]
where $\lambda_0$ will be chosen below. Recall from (\ref{eq:hkub}) that $p_{2rV(z,r)}^{\nu}(z,z) \leq 2V(z,r)^{-1}$. Thus, for $r=h^{-1}(t/4\lambda)$ with $\lambda$ contained in the set defining $\Lambda_{w',y}$, we have $2rV(z,r)\leq t/2$ for $z \in\{w',y\}$, and so
\begin{equation}\label{zwbound}
p_{t/2}^\nu(z,z)\leq \frac{2}{V(z,r)}\leq \frac{2(\log \lambda)^{\frac{1-\alpha}{\alpha}} }{v(h^{-1}(t/4\lambda))^{1/\alpha}}\leq \frac{c_1(\log \lambda)^{\frac{1-\alpha}{\alpha}}\lambda^{\frac{\gamma}{\alpha+\gamma}} h^{-1}(t)}{t},\qquad \forall z \in\{w',y\},
\end{equation}
where the final inequality is a consequence of property (2) of Lemma~\ref{lem:hprops} (see \eqref{eq:polygrowth} for the definition of $\gamma$). In particular, this implies, for any $\varepsilon>0$,
\[\mathbf{E}I_1\leq\int_{\partial B}\int_{\partial H}\mathbf{E}\left(P_x^\nu\left(T_{B}\leq t/2,\:X^\nu_{T_{B}}\in dw\right) P_w\left(X_{T_{H}}\in dw'\right)\frac{c_1\Lambda_{w',y}^{\frac{\gamma}{\alpha+\gamma}+\varepsilon} h^{-1}(t)}{t}\right).\]
Observe that our assumption $at\leq h(d(x,y)^{1/\beta})$ implies that $B_R(z,h^{-1}(t/4\lambda))\cap B=\emptyset$ for $\lambda\geq \lambda_0$ uniformly in $z\in {\partial H}\cup\{y\}$ (and moreover this choice of $\lambda_0$ can be made independently of $x$ and $y$). It follows that $\Lambda_{w',y}$ is $\nu|_{B^c}$ measurable. Since $P_x^\nu(T_{B}\leq t/2,\:X^\nu_{T_{B}}\in dw) P_w(X_{T_{H}}\in dw')$ is $\nu|_{B}$ measurable, we obtain that
\begin{equation}\label{i1b}
\mathbf{E}I_1\leq \int_{\partial B}\mathbb{P}_x\left(T_{B}\leq t/2,\:X^\nu_{T_{B}}\in dw\right)\int_{\partial H}P_w\left(X_{T_{H}}\in dw'\right)\frac{c_1h^{-1}(t)\mathbf{E}\left(\Lambda_{w',y}^{\frac{\gamma}{\alpha+\gamma}+\varepsilon} \right)}{t}.
\end{equation}
We next observe that the stable tail estimates of (\ref{eq:subtail}) and (\ref{eq:uptail}) imply that
$\mathbf{P}(\Lambda_{w',y}\geq \lambda)\leq c_2\lambda^{-\alpha}$. Hence $\mathbf{E}(\Lambda_{w',y}^{\frac{\gamma}{\alpha+\gamma}+\varepsilon} )\leq c_3$ whenever
$\frac{\gamma}{\alpha+\gamma}+\varepsilon<\alpha$. In particular, we can always choose $\varepsilon$ small enough so that this is satisfied whenever
\begin{equation}\label{alpha0def}
\alpha>\alpha_c:=\frac{\sqrt{\gamma^2+4\gamma}-\gamma}{2}\in(0,1).
\end{equation}
Returning to (\ref{i1b}), this yields that, for $\alpha>\alpha_c$,
\begin{eqnarray*}
\mathbf{E}I_1&\leq& \int_{\partial B}\mathbb{P}_x\left(T_{B}\leq t/2,\:X^\nu_{T_{B}}\in dw\right)\int_{\partial H}P_w\left(X_{T_{H}}\in dw'\right)\frac{c_4h^{-1}(t)}{t}\\
&=& \frac{c_4h^{-1}(t)}{t}\mathbb{P}_x\left(T_{B}\leq t/2\right).
\end{eqnarray*}
The result is thus a consequence of Theorem \ref{annexit}(a).

Finally, in the case when $N=0$, we start by applying the Cauchy-Schwarz inequality to deduce that
$p_{t}^\nu(x,y)\leq p_{t}^\nu(x,x)^{1/2}p_{t}^\nu(y,y)^{1/2}$. Hence, if we define $\Lambda_{x,y}$ similarly to above, then we can proceed as at (\ref{zwbound}) to deduce, for any $\varepsilon>0$,
\[p_{t}^\nu(x,y)\leq \frac{c_1\Lambda_{x,y}^{\frac{\gamma}{\alpha+\gamma}+\varepsilon} h^{-1}(t)}{t}.\]
Taking expectations as before yields the result in the case.
\end{proof}

\subsection{Proof of lower bounds}

Throughout this subsection, we suppose that UVD and GMC hold, so that, in particular, $d$ is a geodesic metric. We start by checking the lower heat kernel bound using a standard chaining argument.

\begin{proof}[Proof of Theorem \ref{annthm}(b)] Set $N=N(a)$. Suppose that $N=0$, meaning that $at>h(d(x,y)^{1/\beta})$. By \eqref{eq:hcont} and \eqref{eq:dformhkbd}, we have that
\begin{equation}\label{eee}
\left|p_t^\nu(x,y)-p_t^\nu(x,x)\right|^2\leq \mathcal{E}\left(p_t^\nu(x,\cdot),p_t^\nu(x,\cdot)\right)R(x,y)\leq t^{-1}p_{t}^\nu(x,x)R(x,y).
\end{equation}
In particular, since $R(x,y)^\beta \leq c_1 d(x,y)$, we have that
\begin{equation}\label{firstr}
p_t^\nu(x,y)\geq p_t^\nu(x,x)\left(1-\sqrt\frac{c_2 d(x,y)^{1/\beta}}{tp_t^\nu(x,x)}\right)\geq p_t^\nu(x,x)\left(1-\sqrt\frac{c_3 h^{-1}(at)}{tp_t^\nu(x,x)}\right).
\end{equation}
Now, consider the event $A_\lambda:=\{\lambda^{-1}v(r/\lambda)\leq V(x,r/2\lambda)\leq V(x,r/\lambda)\leq \lambda v(r/\lambda)\}$. On $A_\lambda$, we obtain from Lemma \ref{lem:diaglb} and UVD that $p_t^\nu(x,x)\geq c(\lambda)h^{-1}(t)/t$, where $c(\lambda)$ is a deterministic function of $\lambda$ taking values in $(0,\infty)$. Inserting this bound into (\ref{firstr}) and taking expectations thus yields
\[\mathbf{E}p_t^\nu(x,y)\geq  \frac{c(\lambda)h^{-1}(t)}{t}\left(1-\sqrt\frac{c_3 h^{-1}(at)}{c(\lambda)h^{-1}(t)}\right)\mathbf{P}(A_\lambda).\]
By (\ref{eq:subtail}) and \eqref{eq:uptail}, it is possible to choose $\lambda$ large enough so that $\mathbf{P}(A_\lambda)>1/2$. Moreover, given $\lambda$, taking $a$ small ensures that $c_3 h^{-1}(at)/c(\lambda)h^{-1}(t)<1/4$, and this is enough to complete the proof in this case.

We now turn our attention to the case when $N\geq 1$. Let $\Delta=d(x,y)/N$, and let $(x_i)_{i=0}^N$ be points on a geodesic from $x$ to $y$ such that $d(x_i,x_{i+1})=\Delta$. For $i=0,\dots,N$, let $B_i:= B_d(x_i,\Delta/4)$, $B_i':= B_d(x_i,c_1\Delta)$, where $c_1>1$ is a constant that will be chosen below. Furthermore, write $B:=\cup_{i=1}^NB_i$, and set
$A_1:=B_1'\backslash B$, and $A_i:=B_i'\backslash (A_{i-1}\cup B)$, $i=2,\dots,N$. Note that $A_1,\dots,A_N,B_1,\dots,B_N$ are disjoint. We define associated events
\[E(A_i):=\left\{\nu(A_i)\leq v(r)^{1/\alpha}\right\},\qquad E(B_i):=\left\{\nu(B_i)\in v(r)^{1/\alpha}\left[2,4\right]\right\},\]
where $r:=\Delta^{1/\beta}$. We have that
\[\mathbf{P}\left(E(A_i)\right)\geq \mathbf{P}\left(\nu(B_i')\leq v(r)^{1/\alpha}\right)=
\mathbf{P}\left(\mathcal{L}_{\mu(B_i')}\leq v(r)^{1/\alpha}\right)
=\mathbf{P}\left(\mathcal{L}_{\mu(B_i')/v(r)}\leq 1\right)\geq c_2\]
where the final inequality holds because $\mu(B_i')\leq c_3v(r)$ (by UVD and GMC). Moreover, since $\mu(B_i)/v(r)\in [c_4,c_5]$, we similarly have that
\[\mathbf{P}\left(E(B_i)\right)= \mathbf{P}\left(\mathcal{L}_{\mu(B_i)/v(r)}\in[2,4]\right)
\geq\inf_{t\in[c_4,c_5]}\mathbf{P}\left(\mathcal{L}_{t}\in[2,4]\right)\geq c_6.\]
Thus we conclude that, if $E:=\cap_{i=1}^N (E(A_i)\cap E(B_i))$, then
\begin{equation}\label{a-1}
\mathbf{P}\left(E\right)\geq e^{-c_7N}.
\end{equation}

For the remainder of the proof, we assume that $E$ holds, and establish a lower heat kernel bound on this event. First note
\begin{equation}\label{a0}
p_t^\nu(x,y)\geq \int_{B_1}\nu(dy_1)\dots\int_{B_{N-1}}\nu(dy_{N-1})\prod_{i=1}^Np_{t/N}^\nu(y_{i-1},y_i),
\end{equation}
where $y_0=x$ and $y_N=y$. Note that, by proceeding similarly to (\ref{firstr}), we have that
\begin{equation}\label{a11}
p^\nu_{t/N}(y_{i-1},y_i)\geq p^\nu_{t/N}(y_{i},y_i)\left(1-\sqrt{\frac{c_8rN}{tp^\nu_{t/N}(y_{i},y_i)}}\right).
\end{equation}
Moreover, by Lemma \ref{lem:diaglb}, it holds that
\begin{equation}\label{a12}
p^\nu_{t/N}(y_{i},y_i)\geq \left(\frac{c_9V(y_i,\tilde{r}/2)}{V(y_i,\tilde{r})}\right)^2\frac{1}{V(y_i,\tilde{r})},\qquad \forall\tilde{r}<R_F/2,\:t/N\leq \frac12 c_9 \tilde{r}V(y_i,\tilde{r}/2).
\end{equation}
Now, choose $\tilde{r}:=c_{10}r$, with $c_{10}$ large enough so that $B_R(y_i,\tilde{r}/2)\supseteq B_i$, noting in particular this implies (on $E$)
\begin{equation}\label{a13}
V(y_i,\tilde{r}/2)\geq 2v(r)^{1/\alpha}.
\end{equation}
Then let $c_1$ be large enough so that $B_i'\supseteq B_R(y,\tilde{r})$ for all $y\in B_i$. Since $B_i'\cap B_j'\neq \emptyset$ only if $|i-j|\leq 2c_1$, and $B_i'\cap B_j \neq \emptyset$ can only occur if $|i-j|\leq c_1+1/4$, it must be the case (on $E$) that
\begin{equation}\label{a14}
V(y_i,\tilde{r})\leq 2(2c_1+1)v(r)^{1/\alpha}+2(c_1+1/4+1)v(r)^{1/\alpha}\leq c_{11}v(r)^{1/\alpha}.
\end{equation}
Combining (\ref{a11}), (\ref{a12}), (\ref{a13}) and (\ref{a14}), we thus obtain
\[p^\nu_{t/N}(y_{i-1},y_i)\geq \frac{c_{12}}{v(r)^{1/\alpha}}\left(1-\sqrt{\frac{c_8h(r)N}{c_{12}t}}\right)\geq \frac{c_{13}}{v(r)^{1/\alpha}},\]
where we used that $at/N\leq h(\Delta^{1/\beta})\leq c_{14}h(r)$ and $at/(N+1)\geq h((N/(N+1))^{1/\beta}\Delta^{1/\beta})\geq c_{15}h(r)$, and have adjusted $a$ to be suitably small so that the constant $c_{13}$ is strictly positive. Returning to (\ref{a0}), and recalling (\ref{a13}), this implies that on $E$ we have that
\[p_t^\nu(x,y)\geq \frac{c_{16}}{v(r)^{1/\alpha}}e^{-c_{17}N}\geq \frac{c_{18}h^{-1}(t)}{t}e^{-c_{19}N}.\]
The result follows from this and the estimate at (\ref{a-1}).
\end{proof}

Finally for this section, we prove the lower exit time bound.

\begin{proof}[Proof of Theorem \ref{annexit}(b)] Let $D<D_F/4$, and suppose $N=N(a)\geq 4$, where $a$ is the constant of Theorem \ref{annthm}(b). Let $y\in B_d(\rho,3D)\backslash B_d(\rho,2D)$, and $\tilde{N}=\tilde{N}(a)$ be defined similarly to (\ref{nadef}) from $d(\rho,y)$ and $T$. In particular, it follows from the fact that $d(\rho,y)\geq D$ that $\tilde{N}\geq N\geq 4$. This implies
$B_{\tilde{N}-1}\subseteq B_d(\rho,D)^C$, where $B_{\tilde{N}-1}$ is defined analogously to the definition of $B_{N-1}$ in the previous proof. In particular, appealing to the estimates deduced in the latter argument, we have
\[P_\rho^\nu\left(T_D\leq T\right)\geq P_\rho^\nu\left(X^\nu_T\in B_{\tilde{N}-1}\right)\geq
\int_{B_1}\nu(dy_1)\dots\int_{B_{\tilde{N}-1}}\nu(dy_{\tilde{N}-1})\prod_{i=1}^{\tilde{N}-1}p_{t/\tilde{N}}^\nu(y_{i-1},y_i)\geq
e^{-c_1\tilde{N}}\]
on $E$ (where this event is now defined with terminal points $\rho$ and $y$). Since $\tilde{N}\leq N(a')$ for suitably large $a'$, this completes the proof when $N\geq 4$. If $N<4$, then this simply implies $aT/4\geq h((D/4)^{1/\beta})$, i.e.\ $T\geq c_2 h(D^{1/\beta})$, and so
\begin{equation}\label{equationtoest}
P_\rho^\nu\left(T_D\leq T\right)\geq P_\rho^\nu\left(T_D\leq c_2 h(D^{1/\beta})\right).
\end{equation}
Now, by choosing $a''$ small enough, we obtain $a''c_2h(D^{1/\beta})/4\leq h(({D}/{4})^{1/\beta})$. Moreover, we have that $v(D^{1/\beta})\geq c_3 k v( (D/k)^{1/\beta})$ (by UVD and GMC), and so
\[ h\left(\left(\frac{D}{k}\right)^{1/\beta}\right) \leq c_4 k^{-1-\beta} h\left({D}^{1/\beta}\right)\leq \frac{a''c_2h(D^{1/\beta})}{k}\]
for $k\geq k_0$. This implies we can estimate the right-hand side of (\ref{equationtoest}) below by $e^{-c_5k_0}=c_6>0$ by applying the first part of the proof, and the result follows by adjusting the constants.
\end{proof}

\section{Quenched off-diagonal heat kernel estimates in one dimension}\label{qoffd}

Here we show that it is possible to obtain well-behaved off diagonal estimates as the randomness in the environment is averaged over the path between points. We start by estimating the short time tail of the exit time distribution and then use this to establish our heat kernel estimates. We will work in one-dimension only for this part. In this case, the resistance metric is equal to the Euclidean distance, and we write $B(x,r)=B_R(x,r)$. For our first result, a hitting time estimate, we use the notation
\[\tau_x:=\inf\{t>0:\:X^{\nu}_t=x\}.\]

\begin{thm}\label{qexit}
For the one-dimensional case we fix $D>0$, then there exist constants $c_i, i=1,\dots,4$, such that, $\mathbf{P}$-a.s, there exists a $t_0>0$ such that for all $t<t_0$
\[ c_1 \exp\left(-c_2 \left(\frac{D^{1+1/\alpha}}{t}\right)^{\alpha}\right) \leq P^{\nu}_0\left(\tau_D\leq t\right)\leq c_3 \exp\left(-c_4 \left(\frac{D^{1+1/\alpha}}{t}\right)^{\alpha}\right).\]
\end{thm}

\begin{proof} Define $N:= \sup\{n\geq 1:\:a t/n\leq  (D/n)^{1+1/\alpha} \}$, where $a$ will be chosen in the proof. Note that $N \asymp ({D^{1+1/\alpha}}/{t})^{\alpha}$. Firstly we establish the upper bound. We have a fixed environment $\nu$. We follow the proof of the upper bound of Theorem~\ref{annexit}.
Let $\Delta = D/N$. We set $\sigma_0=0$ and let $\sigma_i = \inf\{t>0:X_t=i\Delta\}$ for
$i=1,\dots,N$ be the successive visits to the points $iD/N$. Using this we have
$\tau_D = \sigma_N = \sum_{i=1}^N (\sigma_i-\sigma_{i-1})$. We note that $\{\sigma_i-\sigma_{i-1}\}_{i=1}^N$ are independent random variables with their distribution depending on the environment.
By Markov's inequality we have
\begin{eqnarray*}
P_0^{\nu}\left(\tau_D \leq t\right) &=& P_0^{\nu}(e^{-\theta \tau_D} \geq e^{-\theta t}) \\
&\leq & e^{\theta t} E_0^{\nu} e^{-\theta \tau_D} \\
&=& e^{\theta t} \prod_{i=1}^N E_0^{\nu} e^{-\theta(\sigma_i - \sigma_{i-1})} \\
&=& e^{\theta t} \prod_{i=1}^N E_{(i-1)\Delta}^\nu e^{-\theta\sigma_1(i)}
\end{eqnarray*}
where $\sigma_1(i)$ is the first hitting time of $i\Delta$ started from $(i-1)\Delta$. As $\sigma_1(i) \geq T_{B((i-1)\Delta,\Delta)}$ we can now use \eqref{exittail} as in the derivation of \eqref{moment0} to obtain
\[ E_{(i-1)\Delta}^{\nu} (e^{-\theta \sigma_1(i)}) \leq E_{(i-1)\Delta}^{\nu} (e^{-\theta T_{B((i-1)\Delta,\Delta)}}) \leq
1-\frac{c_1V((i-1)\Delta,\frac14C_R^2\Delta)}{V((i-1)\Delta,\Delta)} + \frac{1}{\theta \Delta V((i-1)\Delta,\Delta)}. \]
Let $A_i = c_1V((i-1)\Delta,\frac14C_R^2\Delta)/V((i-1)\Delta,\Delta)$ which has the same law as $c_1 V(0,\frac14C_R^2)/V(0,1)$ under $\mathbf{P}$, and let
$B_i=v(\Delta)^{1/\alpha}/V((i-1)\Delta,\Delta)$ which has the same law as $1/V(0,1)$ under $\mathbf{P}$. We can then write
\[P_0^{\nu}\left(\tau_D \leq t\right) \leq  e^{\theta t} \prod_{i=1}^N \left(1-A_i +\frac{B_i}{\theta h(\Delta)}\right) = \exp\left(\theta t - \sum_{i=1}^N \left( A_i - \frac{B_i}{\theta h(\Delta)}\right) \right).\]
We now observe that $A_i$ are positive and bounded random variables and hence all moments exist. We also note that, by the lower tail estimate in \eqref{eq:subtail}, $B_i$ have polynomial moments. Now let $\theta_N = \eta h(\Delta)^{-1}$, then by a standard fourth moment estimate we have for any $\delta>0$, a constant $C$ such that
\[ \mathbf{P} \left(\frac{1}{N}\left| \sum_{i=1}^N\left( A_i-\frac{B_i}{\eta} \right) -  \mathbf{E}\left(A-\frac{B}{\eta}\right)\right| \geq \delta\right) \leq CN^{-2}. \]
A straightforward Borel-Cantelli argument yields
\begin{equation} \frac{1}{N} \sum_{i=1}^N\left( A_i-\frac{B_i}{\eta} \right)\to \mathbf{E}\left(A-\frac{B}{\eta}\right), \;\; \mathbf{P}\mbox{-a.s.}\label{slln}
\end{equation}
Thus, if $\eta$ is chosen large, then there exist constants $N_0$ and $C>0$ such that
\[ \sum_{i=1}^N \left( A_i - \frac{B_i}{\theta_N h(\Delta)} \right) \geq C N,\]
for $N>N_0$, $\mathbf{P}$-almost surely. Using this we have that for $N>N_0$
\begin{eqnarray*}
 P_0^{\nu}\left(\tau_D \leq t\right) &\leq &  \exp\left(\theta_N t - CN\right)\\
 &=& \exp\left( \eta \left(\frac{N}{D}\right)^{1+1/\alpha} t -CN \right) \\
&\leq & \exp\left(-CN\left(1-\frac{\eta N^{1/\alpha} t}{CD^{1+1/\alpha}}\right)\right) \\
&\leq & \exp\left(-(C-a\eta)N\right).
\end{eqnarray*}
Thus by choosing $a$ small we have the upper bound.

For the lower bound we consider the following events. Recall $X$ denotes the one-dimensional Brownian motion that we are time changing,
and $\{L_t(x):0\leq t\}$ is its local time process at $x$. Consider the events
\[ E_i = \{ \{X_t:\sigma_{i-1}\leq t \leq \sigma_i\} \subset [(i-2)\Delta, i\Delta] \} \cap \{ \|L^._{\sigma_i}-L^._{\sigma_{i-1}}\|_{\infty}
\leq \frac{t}{NV((i-1)\Delta,\Delta)} \}. \]
By construction of these events we have $\cap_{i=1}^N E_i \subset \{ \tau_D \leq t \}$.

By the strong Markov property, the events $E_i$ are independent (given $\nu$), and hence $P_0^{\nu}(\tau_D \leq t) \geq \prod_{i=1}^N P_0(E_i)$, where we note $P_0(E_i)$ is a random variable depending on $V((i-1)\Delta,\Delta)$. For the Brownian motion to remain in the interval $((i-2)\Delta,\Delta)$ up to time $\sigma_i$ we must we have the process exiting the interval
at $i\Delta$ and hence $P_{(i-1)\Delta}(X_t \in ((i-2)\Delta,i\Delta]: t\leq \sigma_i) = 1/2$, and hence $P_0^{\nu}(\tau_D \leq t)$ is bounded below by
\[2^{-N} \prod_{i=1}^N P_0\left(\|L_{\sigma_i}-L_{\sigma_{i-1}}\|_\infty \leq \frac{t}{N V((i-1)\Delta,\Delta)}\,\vline\,\{X_t:\sigma_{i-1}\leq t \leq \sigma_i\} \subset [(i-2)\Delta, i\Delta]\right).\]
On the event $E_i$ we have Brownian motion in the interval $((i-2)\Delta, i\Delta)$ and hence on $E_i$ we have $\|L_{\sigma_i}-L_{\sigma_{i-1}} \|_{\infty}$ is equal in distribution to $\|L_{\tau_1\wedge\tau_{-1}} \|_{\infty}$.
Thus we want to consider the random variable $\|L_{\tau_1\wedge\tau_{-1}} \|_{\infty} $.

\begin{lemma}
There exists a constant $c$ such that
\[ P_0(\|L_{\tau_1\wedge\tau_{-1}} \|_{\infty} \leq \lambda) \geq c\lambda, \;\; 0<\lambda <1. \]
\end{lemma}

\begin{proof}
 Clearly
$ \|L_{\tau_1\wedge\tau_{-1}} \|_{\infty} \leq \sup_{x\in [-1,1]} L_{\tau_1}(x)$ and thus
\[ P_0(\|L_{\tau_1\wedge\tau_{-1}} \|_{\infty} \leq \lambda) \geq P_{-1}(\sup_{x\in [-1,1]} L_{\tau_1}(x) \leq \lambda). \]
Thus we just consider $ P_0(\sup_{x\in [0,2]} L_{\tau_2}(x) \leq \lambda)$. By the first Ray-Knight Theorem we have
$L_{\tau_2}(2-x)=Z_x$, for $0\leq x\leq 2$, where $Z$ is a square Bessel process of index 2. Hence
\[ P_0(\|L_{\tau_1\wedge\tau_{-1}} \|_{\infty} \leq \lambda) \geq P_0(\sup_{x\in [0,2]} Z_x \leq \lambda). \]
The square Bessel is the square of the radius of a two-dimensional Brownian motion $W=(W^1,W^2)$ and hence
\[  P_0(\sup_{x\in [0,2]} Z_x \leq \lambda) = P_0(\max_{0\leq x\leq 2} |W_x| \leq \sqrt{\lambda}). \]
By scaling, considering a box inside the ball of radius 1 and using the reflection principle, we can write this as
\begin{eqnarray*}
P_0(\sup_{x\in [0,2]} Z_x \leq \lambda) &=& P_0(\max_{0\leq x\leq 2/\lambda} |W_x| \leq 1) \\
 &\geq & P_0(\max_{0\leq x\leq 2/\lambda} W^1_x \leq 1/\sqrt{2}, \max_{0\leq x\leq 2/\lambda} W^2_x \leq 1/\sqrt{2}) \\
&=& P_0(\max_{0\leq x\leq 2/\lambda} W^1_x \leq 1/\sqrt{2})^2  \\
&=& 4P_0(0\leq W_{2/\lambda} \leq 1/\sqrt{2})^2 \\
&=& \frac{2}{\pi}(\int_0^{\sqrt{\lambda}/2} \exp(-y^2/2) dy)^2.
\end{eqnarray*}
Finally, from the asymptotics of the integral, for small $\lambda$, there is a constant $c$ such that
\[ P_0(\|L_{\tau_1\wedge\tau_{-1}} \|_{\infty} \leq \lambda) \geq P_0(\sup_{x\in [0,2]} Z_x \leq \lambda) \geq c \lambda. \]
\end{proof}

We now continue the proof of Theorem~\ref{qexit}. With this lemma we can obtain our lower bound result as
\begin{equation}
P_0^{\nu}(\tau_D \leq t)  \geq  \prod_{i=1}^N \frac{c t}{2 N V((i-1)\Delta,\Delta)}
= \exp\left(-\sum_{i=1}^N\log\left(\frac{2N V((i-1)\Delta,\Delta)}{c t} \right) \right).
  \label{eq:tdlb}
\end{equation}
We note that by scaling and choice of $N$ we have $V_i^N:={2NV((i-1)\Delta,\Delta)}/{ct} \leq c' V(0,1)$ in distribution. Hence, as the logarithmic moments of the volume will exist for all $\alpha>0$ we proceed in the same way as for the convergence result \eqref{slln}
to see that there is a $C>0$ such that
\[ \sum_{i=1}^N\log_+\left(V_i^N \right) \leq C N, \;\; \mathbf{P}\mbox{-a.s.} \]
Using this in \eqref{eq:tdlb} we have the lower bound.
\end{proof}

We now combine this with the on-diagonal part to establish our quenched heat kernel estimate
Theorem~\ref{thm:fullqhk}. Firstly we give a result connecting the tail of the exit time distribution with the lower off-diagonal heat
kernel estimate.
{\lemma\label{lowergen} $\mathbf{P}$-a.s., for every $x,y\in F$, $t>0$,
\[p_t^\nu(x,y)\geq P_x^\nu\left(\tau_y\leq t/2\right)p_t^\nu(y,y)\]}

\begin{proof} The following proof holds $\mathbf{P}$-a.s. Fix $x,y\in F$, $t>0$. Since the heat kernel is continuous, we have that
\begin{eqnarray*}
p_t^\nu(x,y)&=&\lim_{\varepsilon\rightarrow 0}\frac{1}{V(y,\varepsilon)}P^\nu_x\left(X_t^\nu\in B_R(y,\varepsilon)\right)\\
&\geq&  \lim_{\varepsilon\rightarrow 0}\frac{1}{V(y,\varepsilon)}P^\nu_x\left(\tau_y\leq t/2\right)\inf_{t/2\leq s\leq t}P^\nu_y\left(X_s^\nu\in B_R(y,\varepsilon)\right)\\
&=&P^\nu_x\left(\tau_y\leq t/2\right) \lim_{\varepsilon\rightarrow 0}\frac{1}{V(y,\varepsilon)}\inf_{t/2\leq s\leq t}\int_{B_R(y,\varepsilon)}p_s(y,z)\nu(dz).
\end{eqnarray*}
Now, proceeding as at (\ref{eee}), we have that, for any $z\in B_R(y,\varepsilon)$ and $s\in[t/2,t]$,
\[p_s^\nu(y,z)\geq p_s^\nu(y,y)\left(1-\sqrt\frac{\varepsilon}{sp_s^\nu(y,y)}\right)\geq p_t^\nu(y,y)\left(1-\sqrt\frac{2\varepsilon}{tp_t^\nu(y,y)}\right).\]
Thus we have deduced that
\[p_t^\nu(x,y)\geq P^\nu_x\left(\tau_y\leq t/2\right) \lim_{\varepsilon\rightarrow 0}p_t^\nu(y,y)\left(1-\sqrt\frac{2\varepsilon}{tp_t^\nu(y,y)}\right),\]
from which the result follows.
\end{proof}

We are now ready to prove Theorem \ref{thm:fullqhk}.

\begin{proof}[Proof of Theorem \ref{thm:fullqhk}]
In a similar way to the derivation of \eqref{eq:hkubsplit}, conditioning on $\tau_m$ where $m=(x+y)/2$, we have
\begin{eqnarray*}
p_t^{\nu}(x,y) & \leq & \int_{H_{y,x}} \int_0^{t/2} P^{\nu}_x(\tau_m\in ds) p^{\nu}_{t/2-s}(m,z)p^{\nu}_{t/2}(z,y)\nu(dz) \\
& & \qquad + \int_{H_{x,y}}\int_0^{t/2} P^{\nu}_y(\tau_m\in ds) p^{\nu}_{t/2-s}(m,z)p^{\nu}_{t/2}(z,x) \nu(dz) :=I_1 + I_2.
\end{eqnarray*}
For the first term, integrating out $z$ over $\mathbb{R}$ and applying Cauchy-Schwarz (as well as the monotonicity of the on-diagonal part of the heat kernel) yields
\[I_1   \leq  \int_0^{t/2} P^{\nu}_x(\tau_m\in ds) p^{\nu}_{t-s}(m,y) \leq   P^{\nu}_x(\tau_m\leq t) \sqrt{p^{\nu}_{t/2}(m,m)p^{\nu}_{t/2}(y,y)}.\]
The same bound holds for $I_2$ with $x$ and $y$ reversed. Thus applying our local on-diagonal estimate (Theorem \ref{localhk}(1))
and the tail estimate for the exit time distribution (Theorem \ref{qexit}) we have the upper bound.

For the lower bound we just apply the Lemma~\ref{lowergen} along with our estimate for the lower bound on the tail of the exit time distribution (Theorem \ref{qexit}) and the local on-diagonal heat kernel bound (Theorem \ref{localhk}(2)).
 \end{proof}

We also have a Varadhan type estimate.

\begin{cor} In the one-dimensional case, there exist constants $c_1,c_2$ such that, $\mathbf{P}$-a.s.,
\[c_1|D|^{1+1/\alpha}\leq\liminf_{t\rightarrow\infty} -t^{\alpha}\log{P_0^{\nu}(\tau_D \leq t)}
\leq \limsup_{t\rightarrow\infty} -t^{\alpha}\log{P_0^{\nu}(\tau_D \leq t)}\leq c_2 |D|^{1+1/\alpha},\:\:\forall D\in\mathbb{R}.\]
\end{cor}

\begin{proof}
It follows from Theorem \ref{qexit} that the claim holds $\mathbf{P}$-a.s.\ for a countable set of $D$, and hence by monotonicity for all $D$.
\end{proof}

\providecommand{\bysame}{\leavevmode\hbox to3em{\hrulefill}\thinspace}
\providecommand{\MR}{\relax\ifhmode\unskip\space\fi MR }
\providecommand{\MRhref}[2]{%
  \href{http://www.ams.org/mathscinet-getitem?mr=#1}{#2}
}
\providecommand{\href}[2]{#2}

\end{document}